\documentclass[12pt]{amsart}
\usepackage[utf8]{inputenc}
\usepackage[english]{babel}
\usepackage{amsmath,amssymb,amsfonts,amsthm}
\usepackage{amscd} 
\usepackage{mathtools} 
\usepackage{longtable}
\usepackage[noadjust]{cite} 
\usepackage{enumitem}
\usepackage{float}
\usepackage[mathcal]{euscript} 
\usepackage[unicode]{hyperref}
\usepackage{ragged2e}
\usepackage{xfrac} 

\usepackage{graphicx, array}
\graphicspath{{.}{figures/}}
\usepackage[labelfont=small,labelformat=simple]{subcaption}

\usepackage{xcolor}

\theoremstyle{plain}
\newtheorem{theorem}{Theorem}[section]
\newtheorem{lemma}[theorem]{Lemma}
\newtheorem{proposition}[theorem]{Proposition}

\theoremstyle{definition}
\newtheorem{definition}[theorem]{Definition}
\newtheorem{example}[theorem]{Example}
\newtheorem{remark}[theorem]{Remark}

\begin{document}

\renewcommand{\datename}{}

\renewcommand{\abstractname}{Abstract}
\renewcommand{\refname}{References}
\renewcommand{\tablename}{Table}
\renewcommand{\figurename}{Figure}
\renewcommand{\proofname}{Proof}

\title[Symplectic invariants of parabolic orbits and flaps]{$C^\infty$ symplectic invariants of parabolic orbits and flaps in 
integrable Hamiltonian systems}

\author{E. Kudryavtseva$^{\diamond,\star}$}

\author{N. Martynchuk$^{\dagger, \star}$}

\thanks{
Affiliations:\\
$^\diamond$ Faculty of Mechanics and Mathematics, Moscow State University, Moscow 119991, Russia \\
$^\star$ Moscow Center of Fundamental and Applied Mathematics, Moscow, Russia \\
$^\dagger$ Bernoulli Institute for Mathematics, Computer Science and Artificial Intelligence, University of Groningen, P.O. Box 407, 9700 AK Groningen, The Netherlands. \\
\textit{E-mail: eakudr@mech.math.msu.su, n.martynchuk@rug.nl}}

  \begin{abstract}
In the present paper, we consider a smooth $C^\infty$ symplectic classification of Lagrangian fibrations near cusp 
singularities, parabolic orbits and cuspidal tori. We show that for these singularities as well as for an arrangement of
singularities known as a flap, which arises in the integrable subcritical Hamiltonian Hopf bifurcation, the action variables  
form a complete set of $C^\infty$ symplectic invariants. We also give a symplectic classification for parabolic orbits 
in the real-analytic case. Namely, we prove that a complete symplectic invariant in this case is given by a real-analytic function germ in two variables. Additionally, we construct several 
symplectic normal forms in the $C^\infty$ and/or real-analytic categories, including real-analytic right and 
right-left symplectic normal forms for parabolic orbits.

\hspace{-5.7mm} \textit{Keywords}: Liouville integrability; 
Moser's path method; Symplectic geometry; integrable Hamiltonian Hopf bifurcation.
 
\end{abstract}

\maketitle

\section{Introduction and summary of the results}

In this work, we study
the problem of symplectic classification of  integrable systems. Recall that an integrable system is specified by a triple $(M, \omega, F)$, where $(M,\omega)$ is a symplectic manifold and
$$
F = (F_1 = H, \ldots, F_n) \colon M \to \mathbb R^n
$$
is an \textit{integral} (also called an \textit{energy-momentum}) \textit{map}, 
consisting of $n$ almost everywhere independent functions $F_i$ 
that pairwise Poisson commute:
$$\omega(X_{F_i}, X_{F_j})= 0 \quad \mbox{for all} \quad i, j = 1, \ldots, n,$$ 
where $X_{F_i}$ is defined by the rule $\omega(X_{F_i}, \cdot) = - \textup{d}F_i.$ The function $F_1 = H$ is typically the Hamiltonian of the system,  generating the dynamics via the Hamiltonian vector field $X_H$, and the number $n$ is by definition the number of degrees of freedom. 

The integral map $F$ naturally gives rise to a (singular) Lagrangian fibration on $M$ whose fibers are connected
components of the common level sets 
$$F^{-1}(f) := \{F_1 = f_1, \ldots, F_n = f_n\}, \quad f = (f_1, \ldots, f_n) \in \mathbb R^n.$$
One can also write this fibration as the quotient map $\mathcal F \colon M \to B,$ where $B$ is 
the set of connected components of $F^{-1}(f), \ f \in \mathbb R^n,$ equipped with the quotient topology 
\cite{Fomenko1988}. The space $B$ is usually referred to as
the 
\textit{bifurcation complex}   (or the \textit{unfolded momentum domain}) of the system. Note that all of the fibers of
$\mathcal F \colon M \to B$ are invariant under the flows of the Hamiltonian vector fields $X_{F_1}, \ldots, X_{F_n}.$

The problem of
symplectic classification of integrable systems amounts to classifying the corresponding singular Lagrangian fibrations $\mathcal F \colon M \to B$ up to a `symplectic equivalence'. Usually, one considers 
`fiberwise' (= \textit{right-left}) \textit{symplectic equivalence}, where two such fibrations are called equivalent if they are related by a symplectomorphism
sending fibers to fibers. Following \cite{DeVerdiere1979, Francoise1988, Guglielmi2018, Bolsinov2018}, we shall
also consider the \textit{right symplectic equivalence} in this paper.

\begin{definition} \label{definition/equivalence}
Two integrable Hamiltonian systems $(M, \omega, F)$ and $(\tilde M, \tilde{\omega}, \tilde{F})$ are called
\textit{(right-left) symplectically equivalent} if there exists a 
 symplectomorphism 
$$\Phi \colon  (M,\omega) \to (\tilde M, \tilde{\omega})$$ 
and a homeomorphism $g \colon \tilde B \to B$ such
that 
$
\mathcal F = g \circ \tilde{\mathcal F} \circ \Phi.
$

Similarly, two integrable Hamiltonian systems $(M, \omega, F)$ and $(\tilde M, \tilde{\omega}, \tilde F)$ are called  \textit{right symplectically equivalent} if there exists a 
 symplectomorphism 
$$\Phi \colon  (M,\omega) \to (\tilde M, \tilde{\omega})$$ 
such
that 
$
F = \tilde{F} \circ \Phi$. 

In the real-analytic case, the definitions are similar and we require that $\Phi$ is a real-analytic
diffeomorphism; we refer to these equivalences as \textit{real-analytic right-left (resp., right) symplectic equivalence}.
\end{definition}
\begin{remark}
The above notion of the right symplectic equivalence is quite natural and was considered previously 
for example in \cite{DeVerdiere1979, Francoise1988, Guglielmi2018, Bolsinov2018}. It can also serve as an intermediate step in the right-left
symplectic classification problem \cite{Bolsinov2018}, and this is the point of view that we shall adopt in this work.
As a side remark, we note that it can happen that the right and right-left \textit{topological} equivalences of smooth maps actually
coincide \cite{Saeki1989, Kudryavtseva2011, Martynchuk2015_2} 
(a simple example is given 
by germs of holomorphic Morse functions $f \colon \mathbb C^n \to \mathbb C$ at critical points).
\end{remark}

Finally, we note that symplectic classification of integrable systems (both in the right and right-left cases)  can be approached from the 

i) local (in a neighborhood of a singular point or a singular orbit),

ii) semilocal (in a saturated neighborhood of a singular fiber)

iii) or global 
perspectives. \\
For a more in depth introduction to this and related problems, we refer the reader to
\cite{Bolsinov2004, Bolsinov2006}.

A natural set of symplectic invariants of an integrable system is given by its action variables. Assume for the moment that all of the fibers $F^{-1}(f)$ are compact and connected. Then the action variables can be defined
by the Mineur--Arnold formula
$$
I_i = \dfrac{1}{2\pi} \int_{\gamma_i} \alpha, \quad \textup{d}\alpha = \omega, 
$$
where  
$\gamma_i$ are independent homology cycles on a regular fiber $F^{-1}(f)$ continuously depending on $f$ (recall that regular, compact and connected fibers $F^{-1}(f)$ are $n$-dimensional tori by the Arnol'd--Liouville theorem \cite{Arnold1968, Arnold1978}). Any symplectomorphism $\Phi \colon  M \to \tilde M$ that respects the fibrations induced by the integral maps $F$ and $\tilde{F}$ must send the set of action variables
of $(M, \omega, F)$ to some set of action variables of $(\tilde M, \tilde{\omega}, \tilde{F})$. Therefore, the action variables
 are indeed symplectic invariants (up to affine $\mathbb R^n \rtimes \textup{SL}(n,\mathbb Z)$ transformations).

It is known that in many cases, action variables form a \textit{complete} (in the sense of \cite{Bolsinov2018_2}, see also
Theorem~\ref{Main_theorem}) set of symplectic invariants 
(but not always;
see for example the classical work by Duistermaat \cite{Duistermaat1980}). The well
known such examples include 
the so-called \textit{toric systems} (Delzant's theorem \cite{Delzant1988}), focus-focus singularities \cite{Vu-Ngoc2003}
(see also \cite{Pelayo2009, Pelayo2011}),
and 
simple Morse functions $H \colon M^2 \to \mathbb R$ on compact symplectic $2$-surfaces \cite{Vu-Ngoc2011, Dufour1994}.
In this connection, \cite{Bolsinov2018_2} posed the problem  of proving the `completeness' property of action variables for 
a larger class of (singularities of) integrable systems. Since then, a major progress has been made in
the real-analytic setting. In particular, it has been shown in 
\cite{Bolsinov2018} that for
the so-called parabolic orbits and cuspidal tori, which are the simplest examples
of degenerate singularities of two-degree of freedoms systems, action variables are
the only symplectic invariants in the real-analytic category. 
One more case where such a situation occurs (in the real-analytic category) are non-degenerate semilocal singularities satisfying a connectedness 
condition, as has recently been established 
 in \cite{KudryavtsevaOshemkov2021}.

All this leaves open the question of what happens in the smooth $C^\infty$ case. 
We note that this case is different from the analytic situation since some of the crucial 
analytic techniques, such as Hartogs's extension theorem, are no longer available. 
Nonetheless, as it turns out, at least for parabolic orbits and cuspidal tori, the completeness of
the action variables can be established also in the smooth $C^\infty$ category. This is the main contribution of the present work.

Our first central theorem concerns
parabolic orbits (see e.g. \cite{Lerman1994, Efstathiou2012, Bolsinov2018}
for a background material). It shows that in the  $C^\infty$ category, the action variables corresponding to the family of vanishing cycles and to the free Hamiltonian circle action form a complete set of   symplectic invariants in a neighborhood of such an orbit.
More specifically, we prove  (see Section~\ref{sec/cusp}, Theorem~\ref{theorem/parabolic_classification_c})
the following statement.

\begin{theorem} \label{theorem/parabolic_classification}
Consider a pair of two-degree of freedom $C^\infty$  integrable systems $
F_i \colon M_i \to \mathbb R^2, \ i = 1, 2,
$
with periodic orbits $\alpha_i$ of parabolic type. Let $V_i \simeq D^3 \times S^1$ be a sufficiently small neighbourhood of $\alpha_i$. Consider the functions $H_i,J_i$ and $C^\infty$ smooth coordinates $(x_i,y_i,\lambda_i, \varphi_i) \in D^3 \times S^1$ on $V_i$, with $\alpha_i=(0,0,0)\times S^1$, such that 

\begin{itemize}
\item[\rm i)] $H_i = x_i^2 - y_i^3 + \lambda_i y_i$ and $J_i = \lambda_i$ are constant on the connected components of $F_i^{-1}(f)$, moreover $H_i$ is a function of $F_i$;

\item[\rm ii)] $J_i$ is a smooth $2\pi$-periodic first integral\footnote{The existence of a
smooth $2\pi$-periodic integral $J$ and such a smooth preliminary normal form i)--ii) is shown in 
\cite{KudryavtsevaMartynchuk2021} (for the analytic case, see \cite{Zung2000} and \cite{Bolsinov2018}).}.

\end{itemize}
Finally, consider the action variable $I_i^\circ$ on the swallow-tail domain
$$
D_i = F_i(\{\lambda_i > 0,\ H_i^2 < 4(\lambda_i/3)^3,\ y_i < \sqrt{\lambda_i/3}\})
\subset \mathbb R^2
$$
corresponding to the family of vanishing cycles
$\{H_i = \mathrm{const}, \lambda_i = \mathrm{const}, \varphi_i = \mathrm{const}, y_i < \sqrt{\lambda_i/3}\}$ 
such that $I_i^\circ>0$ and $I_i^\circ(f)\to0$ as $f\to F_i(\alpha_i)$.
Then $F_1$ and $F_2$ are symplectically equivalent near the parabolic orbits $\alpha_1$ and $\alpha_2$, if and only if there exists a diffeomorphism germ $g \colon (\mathbb R^2,F_1(\alpha_1)) \to (\mathbb R^2,F_2(\alpha_2))$ that respects the swallow-tail domains, 
$g(D_1)=D_2$, and makes the actions equal,
$I^\circ_1 = I^\circ_2 \circ g  \mbox{ and } J_1 = J_2 \circ g$ on $D_1$.
\end{theorem}

\begin{remark} \label{remark/affine}
Theorem~\ref{theorem/parabolic_classification} can be strengthened as follows (this can be shown using the smoothness of $J$ \cite{KudryavtsevaMartynchuk2021} and non-differentiability of $I^\circ$ on the closure of the swallow-tail domain for $J>0$): symplectically different parabolic orbits are classified by the inequivalent integer
affine structures on the swallow-tail domain
$$
D =  {\{(h,\lambda) \colon \lambda > 0,\ h^2 < 4(\lambda/3)^3\}} 
\subset \mathbb R^2
$$
intersected with a small neighbourhood of the origin, 
where two integer affine structures are equivalent if they can be obtained from one another by a
diffeomorphism germ $g \colon \mathbb R^2 \to \mathbb R^2$ at $0$ respecting the domain $D$. See Figure~\ref{figure/all}, where
the swallow-tail domain $D,$ the \textit{bifurcation diagram} (i.e., the set of the critical values of $(H,J)$), and 
the singular Lagrangian
fibration
given by $(H,J)$ are shown.

\end{remark}

\begin{figure}[htbp]
\begin{center}
\includegraphics[width=0.98\linewidth]{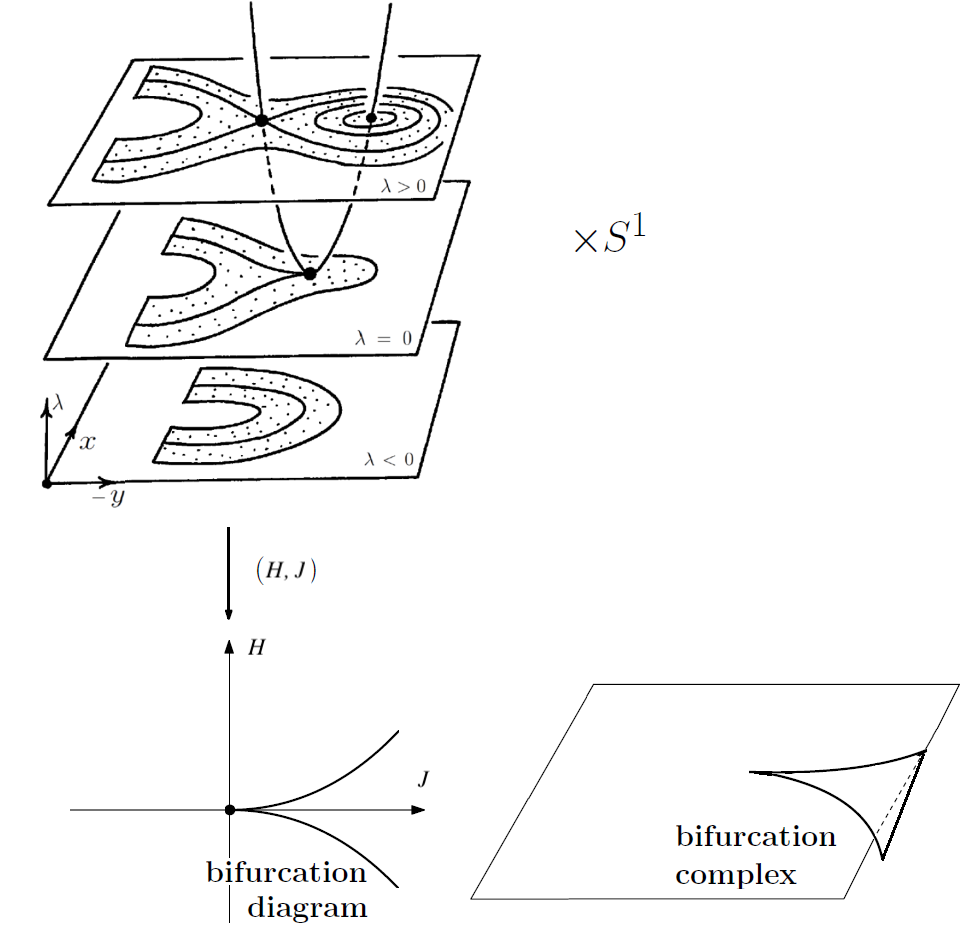}
\caption{The singular Lagrangian
fibration (top), the bifurcation diagram (bottom left) and the bifurcation complex (bottom right) of the energy-momentum map $(H,J)$. The vanishing cycles are small loops
around the `right' branch of the parabola in the 
top figure. The swallow-tail domain $D$ is the region in the bottom left figure
to the right of the semicubical parabola (= the bifurcation diagram).}
\label{figure/all}
\end{center}
\end{figure}

We note that Theorem~\ref{theorem/parabolic_classification} is a result about
smooth $C^\infty$ equivalence, but the method of proof works also in the analytic setting. It is different from the
proof given in \cite{Bolsinov2018}, based on an analytic extension, and allows us to obtain (see Section~\ref{sec/cusp}, Theorem~\ref{theorem_parabolic_normal_form}) the following analytic
normal form result classifying parabolic orbits up to the (real-analytic) right symplectic equivalence.

\begin{theorem} \label{theorem_analytic_right_intro}
Consider a parabolic orbit $\alpha$ of a real-analytic integrable system 
$F = 
(F_1,F_2) 
\colon M \to \mathbb R^2.$ Let the functions $H, J$ and coordinates 
$(x,y,J = \lambda, \varphi)$ be as in Theorem~\ref{theorem/parabolic_classification}, but now real-analytic. 
Then, up to the real-analytic right symplectic equivalence, the  
functions $H, J$ and the symplectic structure $\omega$ 
have the following local normal form near $\alpha$:
$$
H = x^2 - y^3 + \lambda y, \quad J = \lambda, \qquad
\omega_{norm} = c(x^2,y) \textup{d}x \wedge \textup{d}y + \textup{d} J \wedge \textup{d} \varphi,
$$
for some uniquely defined real-analytic function germ $c = c(x^2,y)$ with $c(0,0) > 0.$ Moreover, if $\partial_{F_2}J(F(\alpha))
\ne0$ and the function germ $c = c(x^2,y)$ corresponds to the (uniquely defined) function $H$ having the form $H = \frac{\pm F_1-a(J)}{b^{3/2}(J)}$ (see \cite[Sec.~2]{Bolsinov2018} for an explicit construction of such $H$), then 
the triple $(\Sigma, J, c)$ formed by the local bifurcation diagram $\Sigma$ of $F$ at $\alpha$ and the germs 
$J = J(F_1,F_2)$ 
and $c = c(x^2,y)$ (at $F(\alpha)$ and $(0,0)$, resp.)
classifies the singularity at $\alpha$ up to the real-analytic right symplectic equivalence.
\end{theorem}

We will also obtain a classification up to the right-left symplectic equivalence
in the real-analytic category using a different approach. More specifically, from 
the existence of a real-analytic symplectic normal form in a neighbourhood of a parabolic point,$^2$ we get (cf.\
Section~\ref{sec/cusp}, Theorem~\ref{theorem_analytic_right_left})
the following result.

\begin{theorem} \label{theorem_analytic_right_left_intro}
Let $\alpha$ be a parabolic orbit of a real-analytic integrable system $F = (F_1, F_2) \colon U \to \mathbb R^2$. 
Consider real-analytic coordinates $(\tilde x,\tilde y, \tilde J = \tilde \lambda, \tilde \mu)$ centered at 
some point $P \in \alpha$ such that in these coordinates,
$$\tilde H = \tilde{x}^2 - \tilde{y}^3 + \tilde{\lambda} \tilde{y} \quad \mbox{ and } \quad \tilde{J} = \tilde \lambda$$
are (uniquely defined) real-analytic functions of $(F_1,F_2)$ and the symplectic structure has the canonical form\footnote{The existence of local real-analytic coordinates bringing the fibration to such a symplectic normal form follows from \cite{VarchenkoGivental1982}, see also 
\cite{Garay2004} and Theorem~\ref{theorem_parabolic_point} below.}
$$
\textup{d} \tilde{x} \wedge \textup{d} \tilde{y} + \textup{d} \tilde{\lambda} \wedge \textup{d} \tilde \mu.
$$
Let $J = J(\tilde H,\tilde J)$ be the $2 \pi$-periodic first integral of the system in a neighbourhood of
$\alpha$ such that $J(0,0) = 0$ and $\partial_{\tilde J}J(0,0) > 0.$
Then
the germ $J = J(\tilde H, \tilde J)$ at $(0,0)$ classifies the singularity at $\alpha$ up to real-analytic symplectic equivalence.
Moreover, 
the functions $\tilde H,\tilde J$ and the symplectic structure $\omega$ have the following local normal form near $\alpha$:
$$
\tilde H = u^2 - v^3 + \tilde \lambda v, \quad \tilde J = \tilde \lambda, \qquad
\omega_{norm} = \textup{d}u \wedge \textup{d}v + \textup{d} J(\tilde H,\tilde J) \wedge \textup{d} \psi
$$
in some real-analytic coordinates $(u,v,\tilde J=\tilde\lambda, \psi) \in D^3 \times S^1$ on a neighbourhood of $\alpha$, with $\alpha=(0,0,0)\times S^1$.
\end{theorem} 

Note that Theorem~\ref{theorem_analytic_right_left_intro} 
implies the real-analytic version of
Theorem~\ref{theorem/parabolic_classification} (this was already known in \cite{Bolsinov2018},
so Theorem~\ref{theorem_analytic_right_left_intro}  can be seen as a generalisation of this result).

In particular, the simplest symplectic model for a parabolic orbit corresponds to the trivial germ $c(x^2,y)\equiv1$ in Theorem \ref {theorem_analytic_right_intro} and to the generating function germ $J(\tilde H,\tilde J)\equiv\tilde J$ in Theorem \ref {theorem_analytic_right_left_intro}.

\begin{remark}
We do not know if analogues of Theorems~\ref{theorem_analytic_right_intro} and \ref{theorem_analytic_right_left_intro}
hold also in the smooth setting. 

\begin{itemize} 
\item{In the $C^\infty$ right symplectic equivalence case, the construction shows that we can similarly obtain a smooth function $c(x^2,y)$ such that its Taylor coefficients 
with respect to the variable $x^2$ are symplectic invariants $c_k = c_k(y)$. (Note that the invariants $c_k$ are
function germs, rather than Taylor series.) However, we do not know
if there are more invariants of the right symplectic equivalence in this case (apart from $J = J(F_1,F_2)$ and
the bifurcation diagram).}

\item{For the $C^\infty$ right-left symplectic equivalence, we do not know if a neighbourhood of a parabolic point admits
`Eliasson-type' canonical coordinates $(\tilde x,\tilde y, \tilde J = \tilde \lambda, \tilde \mu)$ as above and whether 
the Taylor series of the function $J = J(\tilde H, \tilde J)$ would then classify the singularity up to
the $C^\infty$ right-left symplectic equivalence. Note that the existence of `Eliasson-type' canonical coordinates $(\tilde x,\tilde y, \tilde J = \tilde \lambda, \tilde \mu)$ is not needed for the proof of Theorem~\ref{theorem/parabolic_classification}.}
\end{itemize}
\end{remark}

 Theorem~\ref{theorem/parabolic_classification} also allows us to obtain a symplectic classification result in terms of
 the integer affine structure in the semi-local setting (that is, in
a neighbourhood of a cuspidal torus) and in a more global situation of an integrable subcritical Hamiltonian Hopf bifurcation (more specifically, neighbourhoods of so-called flaps \cite{Efstathiou2012}), where parabolic orbits appear naturally.  As a specific example, consider the bifurcation diagram of the quadratic spherical pendulum 
\begin{align} \label{eq/energy_momentum_map_qsp}
H &= \frac{1}{2}(p_x^2+p_y^2+p_z^2) + V(z), \\
J &= xp_y - y p_x,
\end{align}
where $V$ is a quadratic potential given by $V = z - z^2$; see Figure~\ref{figure/qsp}.

\begin{figure}[htbp]
\begin{center}
\includegraphics[width=0.95\linewidth]{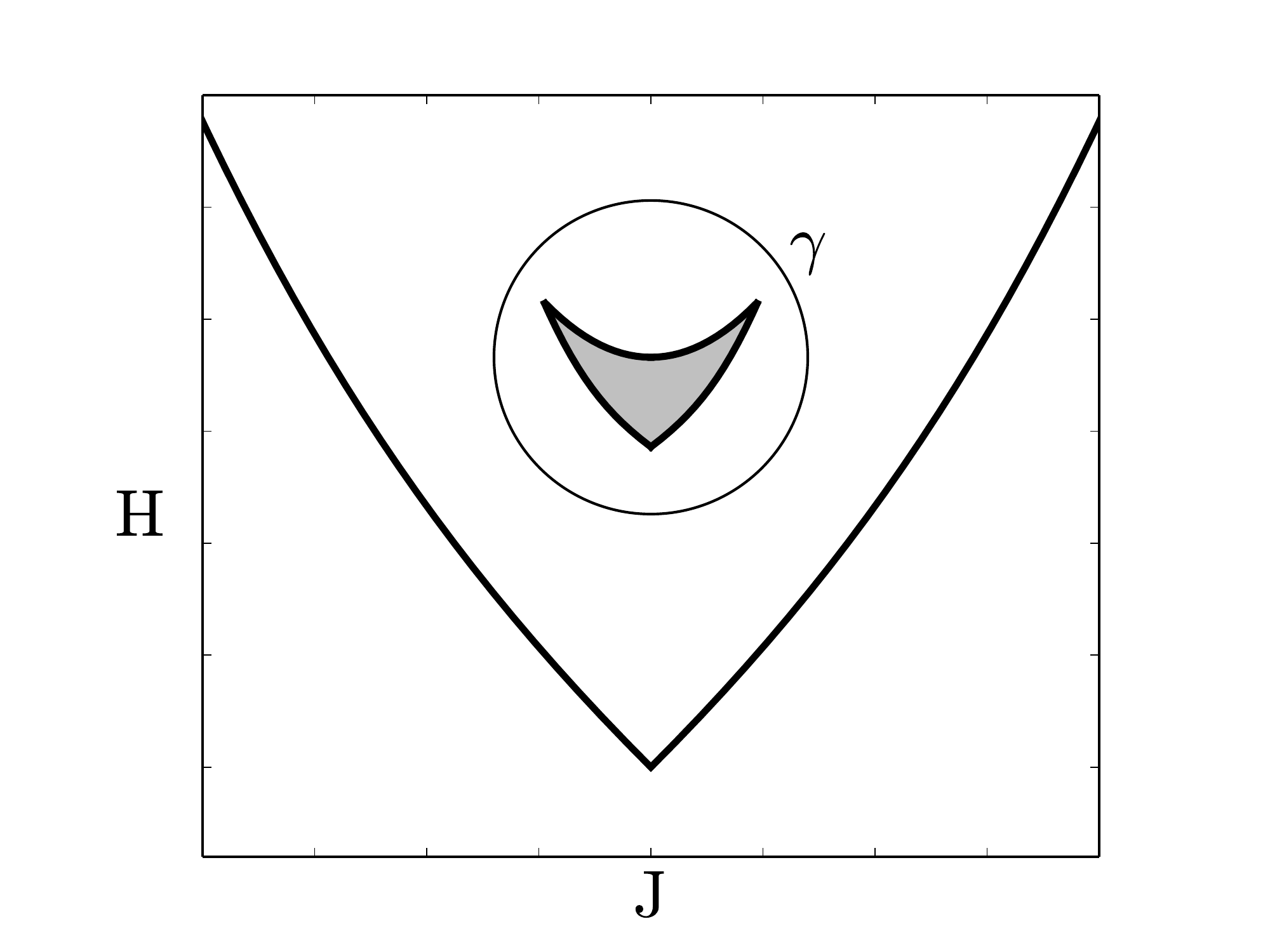}
\caption{The bifurcation diagram of a quadratic spherical pendulum and the flap (shown in gray). The integer affine structure in the interior of the curve $\gamma$ completely classifies 
the system in a neighbourhood of the flap up to symplectic equivalence.}
\label{figure/qsp}
\end{center}
\end{figure}

For a quadratic spherical pendulum, the phase space is the cotangent bundle 
$T^{*}S^2$ and the symplectic structure is canonical, but we may also
consider other symplectic structures on $T^{*}S^2$ such that $H$ and $J$ still Poisson commute. Our results then imply
the following statement. 

\begin{theorem} \label{Main_theorem}
For any two symplectic structures $\omega_1$ and $\omega_2$ such that $H$ and $J$ Poisson commute, the corresponding fibrations
are (smoothly, right-left) symplectically equivalent in a neighborhood of the flap if and only if the corresponding integer affine structures are equivalent, in the sense that there exists a diffeomorphism\footnote{In this paper, by a diffeomorphism between bifurcation complexes $B$ and $\tilde B$ of singular Lagrangian fibrations $\mathcal F \colon M \to B$ and 
$\tilde{\mathcal F} \colon \tilde M \to \tilde B$ we mean a map $g \colon B \to \tilde B$ such that  
$  g \circ \mathcal F = \tilde{\mathcal F} \circ \Phi
$
for some (not necessarily canonical) diffeomorphism $\Phi \colon M \to \tilde M$. } between the corresponding
bifurcation complexes sending one integer affine structure onto the other. 
\end{theorem}


\begin{remark}
The fact that the functions $H$ and $J$ are as in Eq.~\eqref{eq/energy_momentum_map_qsp} is of course not essential,
but at the same time it does not restrict the generality since there is only `one flap' arising from an integrable subcritical Hamiltonian Hopf bifurcation up to a smooth fiberwise
diffeomorphism; see \cite{vanderMeer1985}. A coordinate free version of
 Theorem~\ref{Main_theorem} will appear
later in this paper; see Theorem~\ref{theorem/main2}.
\end{remark}

Our method of proof leading to the above results is heavily based on Moser's path method, which we shall view as a linear PDE in suitable local coordinates bringing the singularity to a (non-canonical) normal form. 
We note that the idea of using Moser's path method in the context of symplectic classification of integrable systems and volume-preserving normal forms is well known and goes back at least to the work \cite{DeVerdiere1979} on a smooth isochore Morse lemma. In particular, this method is used in the work \cite{Bolsinov2018} to give a symplectic normal form  of a cusp singularity in the analytic case of one degree of freedom systems. The perspective of the present work, however, is somewhat different. Considering Moser's path method as
a linear PDE in local coordinates will essentially allow us to reduce the question about `sufficiency' of the action variables to a more algebraic problem of verifying that the linear PDE admits a well-defined and smooth (respectively, analytic) solution. 
As we shall see later in this paper, this simple reformulation of the problem turns out to be quite useful and applicable
also to other (possibly degenerate) singularities of integrable systems.

 We note that the methods used in this paper can be adapted to the analytic and partially also to the finite-differentiable settings, but we shall mainly  be interested in the  smooth $C^\infty$ case.

\section{Preliminaries on the one degree of freedom case} \label{sec/1dof_prelim}

Consider a smooth function $H$ on $\mathbb R^2$. Let $\omega_0$ and $\omega_1, \ \omega_1/\omega_0 > 0,$ be two symplectic forms defined in a neighborhood of the origin $O$ in $\mathbb R^2$.  Assume that $H$ has an isolated singularity at $O$. For the moment, we shall also assume  that locally each level set of $H$ is  the graph of a function $\beta$ of a fixed variable $x$, that is, we assume that locally
\begin{equation} \label{eq/mon}
H^{-1}(h)  = \{(x,y) \colon y = \beta(x,h)\},
\end{equation}
where $(x,y) \in \mathbb R^2$ are some local coordinates around $O = \{x = y = 0\}$ and $\beta$ is smooth outside the point $x = 0, h = H(0,0)$.

Consider the compact region $R(h)$ in $\mathbb R^2$ that is enclosed by the $H = 0$ and $H = h$ level sets and
 two Lagrangian sections $x = \pm \varepsilon;$ see Fig.~\ref{fig/fibration_lambda0}.
Define the area functions (local action variables corresponding to the two symplectic forms) by
$$
I_i(h) = \dfrac{1}{2\pi} \int\limits_{(x,y) \in R(h)} \omega_i.
$$
In this case, the question posed in \cite{Bolsinov2018_2} can be reformulated as follows. Assume that $I_1(h) - I_0(h)$
is a smooth function of $h$. Is it true that there exists a local diffeomorphism $\Phi$ such that
$H \circ \Phi = H$ and $\Phi^{*}(\omega_1) = \omega_0$?

\begin{figure}[htbp]
\begin{center}
\includegraphics[width=\linewidth]{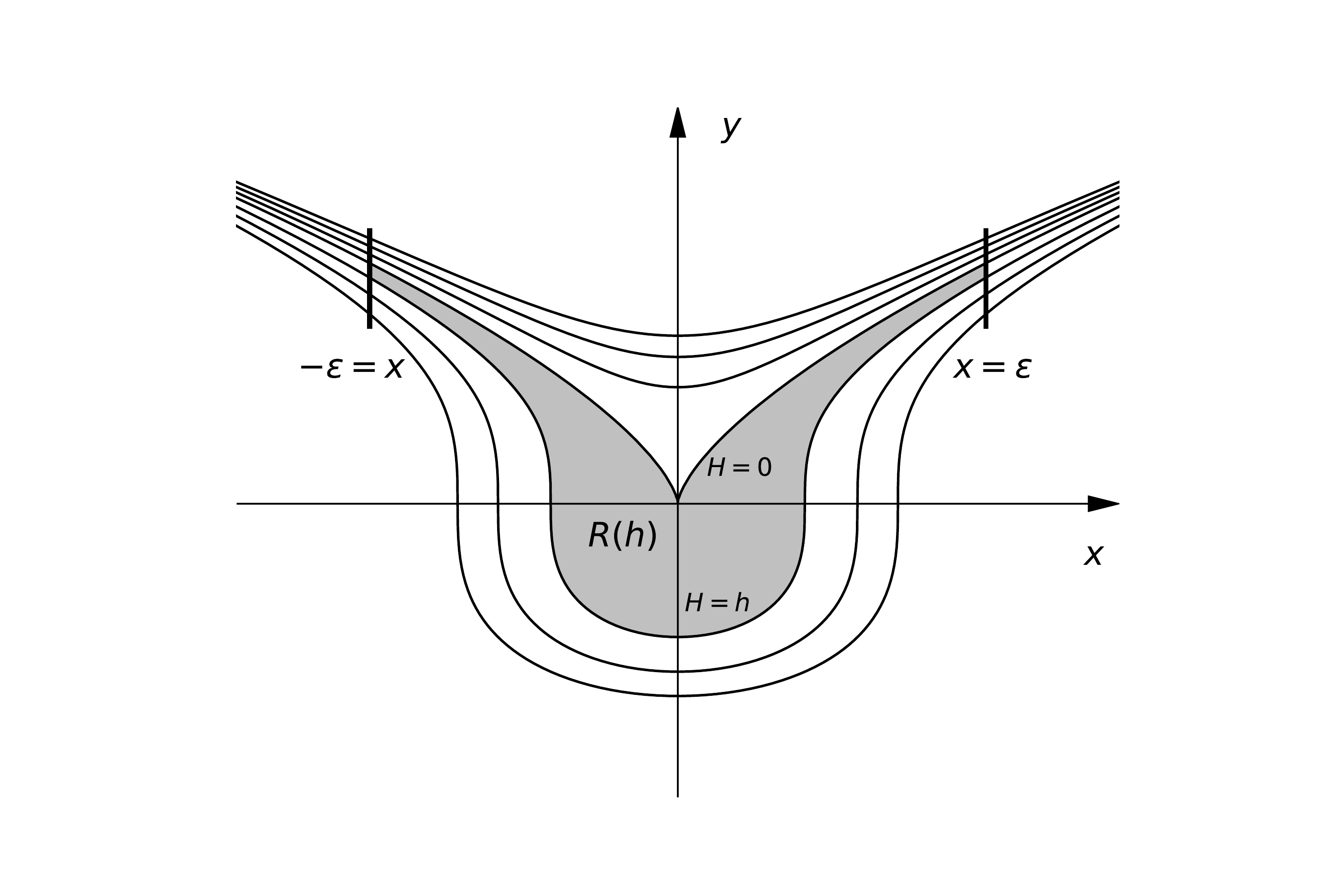}
\caption{The singular Lagrangian
fibration of $H$, sections $x = \pm \varepsilon$ and the region $R(h)$.}
\label{fig/fibration_lambda0}
\end{center}
\end{figure}

A sufficient  condition for the existence of such $\Phi$ is given by an adaptation of Moser's trick, as
described in the following lemma (see also \cite[Lemme principal]{DeVerdiere1979} and \cite[Proposition 2.1 and Theorem 2.1]{Francoise1988}).

\begin{lemma} \label{lemma/move}
Suppose that there exists a smooth function $u$ such that 
\begin{equation} \label{eq/Moser}
\textup{d}u \wedge \textup{d}H = \omega_1 - \omega_0.
\end{equation}
Then there exists a local diffeomorphism $\Phi$ such that 
$H \circ \Phi = H$ and $\Phi^{*}\omega_1 = \omega_0$. Moreover, we can choose $\Phi$ to be
smoothly isotopic to the identity in the class of $H$-preserving diffeomorphisms.
\end{lemma}
\begin{proof}
Let the vector field $X_t$ be defined by the formula
$$i_{X_t} \omega_t = - u\, \textup{d}H,$$ where $\omega_t = \omega_0 + t(\omega_1 - \omega_0), \ t \in [0,1]$.
The required diffeomorphism can then be defined by integrating this vector field $X_t$ from $0$ to $1$, that is, as the solution at $t = 1$ of the equation 
\begin{align*}
 \frac{\textup{d}}{\textup{d}t}\Phi_t = X_t \circ \Phi_t, \quad \Phi_0 = \textup{id}.
\end{align*}
Indeed, under the map $\Phi_t$, we have $\frac{\textup{d}}{\textup{d}t}(\Phi_t^{*}\omega_t) = \Phi_t^{*}(\textup{d} i_{X_t} \omega_t + \frac{\textup{d}}{\textup{d}t} \omega_t)
= 0$ and hence $\Phi_1^*\omega_1= 
\omega_0.$ It is not difficult to check that $\Phi_t$
leaves $H$ invariant. Therefore, $\Phi_1$ is as required.

Let us also give another proof of this fact, with a simple `explicit' formula for an $H$-preserving diffeomorphism $\Phi$.
More specifically, let $\sigma^t_0$ and $\sigma_1^t$ denote the Hamiltonian flows of $H$ w.r.t.\ $\omega_0$ and $\omega_1,$
respectively.
Following \cite[Eq.~(5.1) and (5.2)]{Bolsinov2018}, we can define a required $\Phi$ simply by 
$$\Phi(Q)=\sigma_1^{u(Q)}(Q).$$
Clearly, $\Phi$ preserves $H$, thus it remains to check that $\Phi^*\omega_1=\omega_0$.
Take any regular $\sigma_1^t$-orbit and choose a smooth function $u_1$ on a neighbourhood $V$ of this orbit such that 
$\omega_1|_V=\textup{d}H\wedge\textup{d}u_1$ (i.e.\ $H,u_1$ are Darboux coordinates for $\omega_1|_V$).
In these coordinates, the Hamiltonian flow $\sigma_1^{t}$ has the form
$\sigma_1^{t}(H,u_1) = (H,u_1 + t)$. Hence 
$$\Phi(H,u_1) = \sigma_1^{u(H,u_1)}(H,u_1) = (H,u_1 + u(H,u_1)).$$
It follows that $\Phi^*\omega_1 = \textup{d}H \wedge \textup{d} (u_1 + u) = \omega_1 + \omega_0 - \omega_1 = \omega_0$
on $V$, and hence everywhere. Note that one can alternatively define a
required diffeomorphism $\tilde \Phi$ 
by 
$\tilde \Phi^{-1}(Q)=\sigma_0^{-u(Q)}(Q).$

We also note that a desired diffeomorphism is not unique, and $\Phi_1, \Phi, \tilde\Phi$ constructed above 
may be different from each other.
\end{proof}

It is thus sufficient to find a smooth function $u$ solving Eq.~\eqref{eq/Moser}.
We make the simple observation that Eq.~\eqref{eq/Moser} is simply a linear PDE on the unknown function $u$, for it can be rewritten as
\begin{equation} \label{eq/PDE}
\partial_x u \cdot \partial_y H - \partial_y u \cdot \partial_x H=  g(x,y), 
\end{equation}
where 
$$
\omega_1 - \omega_0 = g(x,y) \textup{d}x \wedge \textup{d}y.
$$
We note that Eq.~\eqref{eq/PDE} can always be solved, but not always can one find a smooth (or even an everywhere well-defined) solution. Indeed, the following formula, coming from the initial condition $\tilde u |_{x = -\varepsilon}=0$, 
\begin{equation} \label{eq/solution}
\tilde{u}(x,h) = \int_{-\varepsilon}^x \dfrac{g(t,\beta(t,h))}{H_y(t,\beta(t,h))} \textup{d}t
\end{equation}
gives rise to a solution $u = \tilde{u}(x,H(x,y))$ of Eq.~\eqref{eq/PDE}; the general solution is then $\tilde{u}(x,H(x,y)) + \alpha(H(x,y))$, where $\alpha$ is a function of $H$.

  Therefore, in this situation, the above question about the action variables amounts to checking whether the smoothness of $$2\pi(I_1'(h) - I_0'(h))
= \frac{\textup{d}}{\textup{d}h} \int_{R(h)} \omega_1-\omega_0 
= \frac{\textup{d}}{\textup{d}h} \int_{R(h)} \textup{d}u \wedge \textup{d}H
= \tilde{u}(\varepsilon,h)
$$
with respect to $h$ is sufficient for $u = u(x,y)$ to be smooth with respect to $x$ and $y$.  In Section~\ref{sec/1dof}, we will show how this can be proven for one degree of freedom elliptic and hyperbolic singularities and also for a cusp
singularity. The two-degree of freedom parabolic case will be discussed later in this paper; see Section~\ref{sec/cusp}.

  We note that the symplectic classification of (one degree of freedom) elliptic and hyperbolic cases is well known in all categories (see \cite{DeVerdiere1979, Eliasson1984, Eliasson1990} for the elliptic  and \cite{Dufour1994} for the hyperbolic case). In the analytic situation, a  parabolic singularity is studied in a neighbourhood of a parabolic point in 
  \cite{VarchenkoGivental1982, Garay2004} and in a neighborhood of a parabolic orbit in \cite{Guglielmi2018, Bolsinov2018}. The symplectic classification of a parabolic singularity in the smooth category that we obtain in this work  seems to be new (already in the one degree of freedom case).

  To close this section, let us discuss some terminology and a simple general procedure that will be useful later on.
  Suppose that we have decomposed the solution $u$ as a sum $u = u_1 + u_2$, where $u_1$ is known to be smooth. Then we can apply Moser's trick using $u_1$, which will have an effect both on Eq.~\eqref{eq/PDE} and \eqref{eq/solution}. More specifically, this Moser's trick, does not change $H$ and transforms the symplectic form $\omega_1$ to another symplectic form $\tilde\omega_1$ 
  such that 
  $$
  \omega_1 - \tilde\omega_1 = \textup{d}u_1 \wedge \textup{d}H.
  $$
  Hence, $\tilde\omega_1 - \omega_0 = \omega_1 - \omega_0 - \textup{d}u_1 \wedge \textup{d}H = \textup{d}u_2 \wedge \textup{d}H.$ Denoting $\tilde\omega_1 - \omega_0$ by $\tilde{g}(x,y) \textup{d}x \wedge \textup{d}y$, we are left with the equation
  \begin{equation*}
\partial_x u_2 \cdot \partial_y H - \partial_y u_2 \cdot \partial_x H=  \tilde{g}(x,y).
\end{equation*}
Moreover, the function $u_2$ can be written as 
\begin{equation*}
u_2 = \tilde u_2(x,H(x,y)), \quad \tilde u_2(x,h) = \int_{-\varepsilon}^x \dfrac{\tilde{g}(t,\beta(t,h))}{H_y(t,\beta(x,h))} \textup{d}t.
\end{equation*}
 It is then left to consider $u_2$ instead of $u$. Note that here linearity of Eq.~\eqref{eq/Moser} (equivalently, Eq.~\eqref{eq/PDE}) is used. 
 
 Essentially, this procedure amounts to changing the symplectic form $\omega_1$ by 
$
  \omega_1 - \textup{d}u_1 \wedge \textup{d}H.
  $  We shall refer to this operation as a \textit{symplectic $u_1$-move}. 
 
 \begin{definition}
 Consider a one degree of freedom Hamiltonian system $(M^2, \omega, H).$ For a smooth function $u$ on
 $M^2$, the transformation
 $$
 \omega \mapsto \omega -  \textup{d}u \wedge \textup{d}H
 $$
 will be called a \textit{symplectic $u$-move} or simply a \textit{symplectic move}. 
 \end{definition}
 
 \begin{remark}
 By Lemma~\ref{lemma/move},  a symplectic $u$-move is always induced by an $H$-preserving diffeomorphism. 
Such a diffeomorphism $\Phi$ can, for instance, be given as the time-1 map of 
 $X_t = \omega_t^{-1}(-u\, \textup{d}H),$ where $\omega_t$ is given by $\omega_t = \omega + (t-1)\textup{d}u \wedge \textup{d}H, \ t \in [0,1].$ Note that $\Phi$ depends on the Hamiltonian $H$ and the symplectic form $\omega$, whereas a symplectic $u$-move depends only on $H$ (for a given $u$). 
 
 In what follows we will mainly be interested in symplectic moves and will not explicitly mention their dependence
on $H$ since it will always be clear from the context what the Hamiltonian function is.

 \end{remark}
 
 The following two simple examples will be used later on.
 
 \begin{example} \label{example/moves}
\mbox{ }

1. Consider the case when $u_1 = a(x)b(y)$, where $a$ and $b$ are smooth. Then the symplectic $u_1$-move transforms $\omega_1$ to the form $\tilde\omega_1$ with
\begin{equation*}
\omega_1 - \tilde\omega_1 = (a'(x)b(y)H_y - a(x)b'(y)H_x) \textup{d}x \wedge \textup{d}y.
\end{equation*} 

2. Let $H = x^2 + f(y)$. Write the symplectic form as 
$$ \omega_1= (xg_o(x^2, y) + g_e(x^2,y))\textup{d}x\wedge \textup{d}y$$ and consider the function
\begin{equation*}
\tilde u_1(y,h) =  -  \int_{0}^y \frac{\sqrt{h-f(t)}}{H_x}g_o(h-f(t),t) \textup{d}t = - \frac12 \int_{0}^y g_o(h-f(t),t) \textup{d}t.
\end{equation*} 
Note that this is essentially the same formula as Eq.~\eqref{eq/solution} with the roles of $x$ and $y$ interchanged, coming from the initial condition $\tilde u_1 |_{y = 0}=0$.
Let $u_1 = \tilde u_1(y,H(x,y))$. Then the  symplectic $u_1$-move transforms $\omega_1$ to its even part $\tilde\omega_1 = g_e(x^2,y)\textup{d}x\wedge \textup{d}y$.

\end{example}

  We remark that  up to this point, the above discussion  in the analytic setting is completely parallel to the $C^\infty$ case. 
  
  \section{Non-degenerate and cusp
singularities in one degree of freedom} \label{sec/1dof}
  
  In this section, we will illustrate the approach outlined in Section~\ref{sec/1dof_prelim} on a few one-degree of freedom systems. We will first review the well-known cases of 1 d.o.f.\  elliptic and hyperbolic singularities and then give a symplectic classification of a 
cusp (= 1 d.o.f.\ parabolic) singularity in the smooth category. The parabolic
singularity in the case of two degrees of freedom will be addressed later in Section~\ref{sec/cusp}.

  \subsection{Elliptic and hyperbolic singularities} 
  
  Let $H$ be a germ of a smooth function on $\mathbb R^2$ with an isolated non-degenerate  critical point at the origin $O$. Here non-degeneracy simply means that the Hessian 
  $$\dfrac{\partial^2 H}{\partial x_i \partial x_j}$$ 
  is non-degenerate at $O$. 
  Consider the local fibration  induced by the function $H$, and let $\omega_1$ and $\omega_0$ be germs of symplectic forms.  We would now like to understand under which conditions there exists a germ of an $H$-preserving symplectomorphism $\Phi$ of  $\mathbb R^2$ sending $\omega_1$ to $\omega_0$.
  
  By the standard Morse lemma, we can assume that up to an additive constant, $H = x^2 \pm y^2$ in some local coordinates near $O$. First, let us assume that we are in the elliptic case, that is, $H = x^2 + y^2$. Then the necessary condition for the existence of 
  an $H$-preserving symplectomorphism $\Phi$ is that the areas (action variables) 
  $$
I_i(h) = \dfrac{1}{2\pi}
\int\limits_{H(x,y) \le h} \omega_i,
$$
where $h > 0$,
  corresponding to the two symplectic forms coincide. This necessary condition is well known to be sufficient in this case.
  
  \begin{proposition}[{\cite{DeVerdiere1979}}] \label{prop/elliptic} 
  Let $H = x^2+y^2$ and  $\omega_1, \omega_0$ be germs of symplectic forms at $O$. Assume that  the action variables corresponding to these forms coincide: $I_1(h) = I_0(h)$. Then there exists a local diffeomorphism $\Phi$ such that
$H \circ \Phi = H$ and $\Phi^{*}(\omega_0) = \omega_1$.
  \end{proposition}
  
  We shall now give a proof of this statement using the method outlined in Section~\ref{sec/1dof_prelim}. First, we show the following normal form result.
  
  \begin{lemma} \label{lemma/1}
  Let $H = x^2+y^2$ and  the symplectic form be given by $\omega = g(x^2,y^2) \textup{d}x \wedge \textup{d}y$. For every $n \in \mathbb N$, there exists a symplectic $u$-move, with 
the function $u$ satisfying $u(x,0)= 0,$ that transforms $\omega$ to the form
  $\omega_{n} = g_n(x^2,y^2) \textup{d}x \wedge \textup{d}y,$ where
  $$g_{n}(x^2,y^2) = \sum_{i=0}^{n} d_i y^{2i} + R(x^2,y^2)y^{2(n+1)}.$$
 In particular, there exists a local diffeomorphism $\Phi$ such that
$H \circ \Phi = H$ and $\Phi^{*}(\omega_n) = \omega$.
  \end{lemma}
  
  \begin{proof}
  Observe that $g(x^2,y^2)$ can be decomposed as 
$$g(x^2,y^2) = \sum_{i = 0}^n g_i(x^2) y^{2i} + r(x^2,y^2)y^{2(n+1)},$$
where $g_i$ and $r$ are smooth functions of their arguments. Similarly, 
$$g_i(x^2) = \sum_{j = 0}^m c_j x^{2j} + r_i(x^2) x^{2(m+1)}.$$ Hence, we can write
$$g(x^2,y^2) = \sum_{i,j}^{n,m} c_j x^{2j} y^{2i} + \sum_{i}^n r_i(x^2)x^{2(m+1)}y^{2i} + r(x^2,y^2)y^{2(n+1)}.$$
Now consider the smooth function $u_{ij} = x^{2j-1}y^{2i+1}$. By Example~\ref{example/moves}, the corresponding symplectic move changes
$g(x^2,y^2)$ by 
$$2(2j-1)x^{2j-2}y^{2i+2} - 2(2i+1)x^{2j}y^{2i}.$$
Similarly, the symplectic $u_{i}$-move with $u_i = r_i(x^2)x^{2m+1} y^{2i+1}$ changes 
$g(x^2,y^2)$ by
$$\tilde{r}_i(x^2)x^{2m}y^{2i+2} - 2(2i+1)r_i(x^2)x^{2(m+1)}y^{2i}.$$
Summing the functions $u_{ij}$ and $u_i$ with appropriate coefficients, we can transform the function
$g(x^2,y^2)$ to the form
$$g_{norm}(x^2,y^2) = \sum_{i}^{n} d_i y^{2i} + R(x^2,y^2)y^{2(n+1)},$$
provided that $m > n.$ It is left to observe that all of the functions $u_{ij}$ and $u_j$ are divisible by $y$.
  \end{proof}
  We are now ready to prove the proposition.
   \begin{proof}
  Write $\omega_1 - \omega_0$ as $\omega_1 - \omega_0 = g(x,y) \textup{d}x \wedge \textup{d}y$ for some smooth germ $g = g(x,y)$. By assumption, 
\begin{equation} \label{eq/actions_zero}
I(h) = \dfrac{1}{2\pi} 
\int\limits_{H(x,y) \le h} g(x,y) \textup{d}x\, \textup{d}y = 0
\end{equation}
for all (small) $h >0$. We shall deduce from this that the two-form
$g(x,y) \textup{d}x \wedge \textup{d}y$ can be written as $
g(x,y) \textup{d}x \wedge \textup{d}y = \textup{d}u \wedge \textup{d}H$
for some smooth germ $u(x,y)$, that is, that $\omega_1$ and $\omega_0$ are related by a $u$-move (this is enough, due to Lemma \ref {lemma/move}). 

Following the strategy explained above, we can first of all assume (see
Example~\ref{example/moves}.2) that $g$ is even with respect to $x$ and $y$, so
that $g(x,y) = g_e(x^2,y^2)$. Next, we rewrite the equation $
g_{e}(x^2,y^2) \textup{d}x \wedge \textup{d}y = \textup{d}u \wedge \textup{d}H$ as the first order PDE
\begin{equation*}
2y u_x - 2x u_y = g_{e}(x^2,y^2).
\end{equation*}
Observe that it admits a solution of the form
\begin{equation*}
u(x,y) = \tilde{u} (x, x^2+y^2), \mbox{ where }\tilde u(x,h) = \int^x_f \dfrac{g_{e}(t^2, h-t^2)}{2\sqrt{h-t^2}} \textup{d}t.
\end{equation*}
Our goal is to show, using Eq.~\eqref{eq/actions_zero}, that for some choice of the initial condition $f$, the function $u$ is smooth.  In the rest of the proof, we shall fix the initial condition $f = -\sqrt{h}$ ($h > 0$). 
Note that Eq.~\eqref{eq/actions_zero} implies $\tilde{u}(\sqrt{h},h)-\tilde{u}(-\sqrt{h},h)=2\pi I'(h)=0$.
We will now prove that for arbitrary large $n$, the function
$\tilde u$  is differentiable at leat $n-1$ times, implying that $\tilde u = \tilde{u}(x,h)$ and hence $u = u(x,y)$ are smooth.

Applying Lemma~\ref{lemma/1}, we can assume that $g$ is of the form 
\begin{equation*} 
g(x,y) = \sum_{i=0}^{n} d_i y^{2i} + R(x^2,y^2)y^{2(n+1)},
\end{equation*}
where $n \in \mathbb N$ is arbitrary large, and write $\tilde{u}(x,h)$ as
\begin{multline} \label{eq/ff0}
2 \tilde{u}(x,h) = \int^x_f  \sum_i d_i (h-t^2)^{(2i-1)/2} \textup{d}t+ \\ \int^x_f R(t^2,h-t^2)(h-t^2)^{(2n+1)/2} \textup{d}t.
\end{multline}
Observe that 
\begin{multline} \label{eq/ff20}
\int_{-\sqrt{h}}^{\sqrt{h}} \sum_i d_{i}  (h-x^2)^{(2i-1)/2} \textup{d}x +  \\ \int_{-\sqrt{h}}^{\sqrt{h}}  R(x^2,h-x^2)(h-x^2)^{(2n+1)/2}\textup{d}x
\end{multline}
is exactly equal to the derivative $2\pi I'(h)$.

We also observe that the second integral in \eqref{eq/ff20} is of order
at least $h^{n +1}$. On the other hand, 
\begin{equation*} 
\int_{-\sqrt{h}}^{\sqrt{h}} (h-x^2)^{(2i-1)/2} \textup{d}x = h^{i} \int_{-1}^{1} (1-x^2)^{(2i-1)/2} \textup{d}x
\end{equation*}
is of order $h^{i}$. Since $I'(h) = 0$, we have that all $d_{i} = 0, \ i { \le } n$. 

Coming back to Eq.~\eqref{eq/ff0}, we infer that $\tilde{u}(x,h)$ is differentiable at least $n-1$ times and hence so is $u(x,y)$. Since $n$ is arbitrary large, we conclude that $u(x,y)$ is smooth.
  \end{proof}

  \begin{remark}[\textit{Normal form of an elliptic singularity}]
In fact, we can go one step further and obtain a (unique) right normal form for a pair $(H = x^2+y^2, \omega)$
up to the right equivalence:
  $$
   (H = x^2+y^2, \omega_{norm}), \quad \omega_{norm} = f(y^2) \textup{d}x \wedge \textup{d}y,
  $$
  where $f$ is a smooth function of $y^2$; cf.\ Lemma~\ref{lemma/1}.  The germ of $f$ (restricted to the half-interval $[0, \infty)$) is a complete invariant of a one degree of freedom elliptic singularity up to the right symplectic equivalence. 
  To show this is indeed the case, consider the equation
    \begin{multline} \label{eq/smooth}
2\pi
I'(h) = \int_{-\sqrt{h}}^{\sqrt{h}} \dfrac{f(h-x^2)}{\sqrt{h-x^2}} \textup{d}x = \\
\int_{-1}^{1} \dfrac{f(h \cdot (1-x^2))}{\sqrt{1-x^2}} \textup{d}x = 
\int_0^h \dfrac{f(t)}{\sqrt{t}\sqrt{h-t}} \textup{d}t
\end{multline}
and for every $n \in \mathbb N$, the equation
 \begin{multline} \label{eq/smooth_2}
2\pi
I^{(n+1)}(h) = \int_{-1}^{1} f_n(h (1-x^2))(1-x^2)^{n-1/2} \textup{d}x =  \\ \frac{1}{h^{n}} 
\int_0^h \dfrac{f_n(t) t^{n-1/2}}{\sqrt{h-t}} \textup{d}t.
\end{multline}
Now observe that $I$ is a smooth function of $h \ge 0$ (this follows from \cite{DeVerdiere1979}, but
can also be proven directly, e.g., from Eq.~\eqref{eq/ff20}).
Because of this, each of the above equations \eqref{eq/smooth}-\eqref{eq/smooth_2} admits (see \cite{Evans1910}) a unique continuous solution $f_n$, which must then be equal to the $n$-th derivative $f^{(n)}$. This proves that there exists a unique smooth solution $f$ of Eq.~\eqref{eq/smooth}. It is left to apply Proposition~\ref{prop/elliptic}.

One can compare the right normal form $
  \omega_{norm} = f(y^2) \textup{d}x \wedge \textup{d}y
  $ to the more usual normal form
  $$
  \omega_{st} = f_{st}(x^2+y^2) \textup{d}x \wedge \textup{d}y.
  $$
given in the work \cite{DeVerdiere1979}. 
  These two normal forms are of course equivalent. Note that the proof that $\omega_{st}$ is a normal form is simpler because of the equality $f_{st}(h) = I'(h)/2$. 
  \end{remark}
  
  In the hyperbolic case $H = x^2 - y^2$, we have the following situation. First observe that the critical level set
  $H = x^2-y^2 = 0$ divides $\mathbb R^2$ into four quadrants (given by $|x| \le y, |y| \le x, y \le -|x|,$ and $x \le -|y|$).
  Let $\omega_1$ and $\omega_0$ be germs of symplectic forms.  Consider the area functions (local action variables corresponding to the two symplectic forms) in one of the quadrants:
$$
I_i(h) = \dfrac{1}{2\pi} \int\limits_{\{(x,y) \colon |y| \le x \le \sqrt{h+y^2}, \ |y|\le \varepsilon\}} \omega_i.
$$

Similarly to the proof of Proposition~\ref{prop/elliptic}, one can prove the following proposition (see \cite{Dufour1994}).
  
  \begin{proposition}[{see \cite{Dufour1994}}] \label{prop/hyperbolic} 
  Let $H = x^2-y^2$ and  $\omega_1, \omega_0$ be germs of symplectic forms at $O$. Assume that  the difference $I_1(h) - I_2(h)$ (corresponding to one of the four quadrants) extends to a smooth function near the origin. Then there exists a local diffeomorphism $\Phi$ such that
$H \circ \Phi = H$ and $\Phi^{*}(\omega_0) = \omega_1$.
  \end{proposition}
  
  \begin{remark}[\textit{Normal form of a hyperbolic singularity}]
  In the hyperbolic situation, similarly to the elliptic case, we can transform a pair $(H = x^2-y^2, \omega)$ to the right normal form $(H = x^2-y^2, \omega_{norm})$ with
  $$
  \omega_{norm} = f(y^2) \textup{d}x \wedge \textup{d}y,
  $$
  where $f$ is smooth. But in this case, it is the Taylor series of $f$ at the origin that is a complete invariant
  of the right symplectic  equivalence. Note that the right normal form $
  \omega_{norm}$ is equivalent to the normal form
  $$
  \omega_{st} = f_{st}(x^2-y^2) \textup{d}x \wedge \textup{d}y,
  $$
  given in the work \cite{DeVerdiere1979}.
  
  \end{remark}
  
  \begin{remark} \label{remark/parametric}
  Using the parametric Morse lemma, it is not difficult to prove parametric versions of the above Propositions~\ref{prop/elliptic} and \ref{prop/hyperbolic}.
  \end{remark}

\subsection{The cusp singularity} \label{subsec/parabolic}

Similarly to the elliptic and hyperbolic cases, consider a fibration of $\mathbb R^2$ induced by the function 
$H = x^2 - y^3$; each $A_2$ singularity reduces to such a form up to an additive constant. Let $\omega_1$ and $\omega_0$ be two symplectic forms defined in a neighborhood of the origin in $\mathbb R^2$. 
Consider the region $R(h)$ in $\mathbb R^2$ enclosed by the $H = 0$ and $H = h$ level sets and
the two Lagrangian sections $x = \pm \varepsilon$; see Fig.~\ref{fig/fibration_lambda0}.
Define the area functions (local action variables corresponding to the two symplectic forms) by
$$
I_i(h) = \dfrac{1}{2\pi}
\int\limits_{(x,y) \in R(h)}
 \omega_i.
$$
\begin{proposition} \label{prop/cusp}
Let $H = x^2 - y^3$ and $\omega_1$ and $\omega_0$ be two symplectic forms. Suppose that the difference $I_1(h) - I_{0}(h)$ of action variables corresponding to these forms is a smooth function of $h$. Then 
$\omega_1/\omega_0 > 0$ and
there exists a local diffeomorphism $\Phi$ such that
$H \circ \Phi = H$ and $\Phi^{*}(\omega_1) = \omega_0$.
\end{proposition}

We will need the following lemma.

  \begin{lemma} \label{lemma/2}
  Let $H = x^2-y^3$ and  $\omega = g(x,y) \textup{d}x \wedge \textup{d}y$ be a symplectic form. For every $n \in \mathbb N$, there exists a symplectic move that transforms $\omega$ to the form
  $$\omega_{n} = g_n(x^2,y) \textup{d}x \wedge \textup{d}y, \ g_{n}(x^2,y) = \sum_{i=0}^{n} d_i y^{i} + R(x^2,y)y^{n+1},$$
  where $d_i = 0$ for $i = 2+3k$.
 In particular, there exists a local diffeomorphism $\Phi$ such that
$H \circ \Phi = H$ and $\Phi^{*}(\omega_n) = \omega$.
  \end{lemma}
  
  \begin{proof}  The proof is parallel to the proof of Lemma~\ref{lemma/1}. First of all, we can assume $g$ is even with respect to $x$. Next, the symplectic $u_{ij}$-move with $u_{ij} = x^{2j-1}y^{i+1}$ changes
$g$ by 
$$3(2j-1)x^{2j-2}y^{i+3} {+}
2(i+1)x^{2j}y^{i}.$$
The symplectic $u_{i}$-move with $u_i = r_i(x^2)x^{2m-1} y^{i+1}$ changes 
$g$ by
$$\tilde{r}_i(x^2)x^{2m-2}y^{i+3} {+}
2(i+1)r_i(x^2)x^{2m}y^{i}.$$
Summing the functions $u_{ij}$ and $u_i$, {for $i\ge -1 $}, with appropriate coefficients, we can transform the function
$g({x}
,y)$ to the form
$$\tilde{g}(x^2,y) = \sum_{i}^{n} d_i y^{i} + R(x^2,y)y^{n+1},$$
where $d_i = 0$ for $i = 2+3k,$
as required.
\end{proof}
  
We are now ready to prove the proposition.

\begin{proof}
As in the proof of Proposition~\ref{prop/elliptic}, let $\omega_1 - \omega_0 = g(x^2,y) \textup{d}x \wedge \textup{d}y$ for some smooth germ $g = g(x^2,y)$. By assumption, 
$$
I(h) = \dfrac{1}{2\pi} \int\limits_{(x,y) \in R(h)}
g(x^2,y) \textup{d}x\, \textup{d}y
$$
is a smooth function of $h$. We shall deduce from this that the  PDE
$$
3y^2 u_x + 2x u_y = - g(x^2,y)
$$
admits a smooth solution of the form
\begin{equation} \label{eq/u}
\tilde u(x,h) = \int_{-\varepsilon}^x \dfrac{g(t^2, (t^2-h)^{1/3})}{-3(t^2-h)^{2/3}} \textup{d}t.
\end{equation}
Note that here we fix the initial condition for $u$ to be $u(-\varepsilon,y) = 0$, where $\varepsilon >0 $ is
as above.
Because of Lemma~\ref{lemma/2}, we can assume
$$g(x^2,y) = \sum_{i=0}^{n} d_i y^{i} + R(x^2,y)y^{n+1},$$
where $d_i = 0$ for $i = 2+3k$. Then, up to a smooth function of $H$, we can write
\begin{multline} \label{eq/ff}
-3u(x,y) = \int_{-\varepsilon}^x  \sum_i d_i (t^2-h)^{(i-2)/3}\textup{d}t|_{h = x^2-y^3}+ \\ 
\int_{-\varepsilon}^{ x} 
R(t^2,(t^2-h)^{1/3})(t^2-h)^{(n-1)/3}\textup{d}t|_{h = x^2-y^3}.
\end{multline}
In particular, the corresponding definite integral
\begin{equation} \label{eq/ff2}
\int_{-\varepsilon}^{\varepsilon} \sum_i d_i  (x^2-h)^{(i-2)/3}\textup{d}x+  \int_{-\varepsilon}^{\varepsilon}  R(x^2,(x^2-h)^{1/3})(x^2-h)^{(n-1)/3}\textup{d}x
\end{equation}
is equal to the derivative $-{6\pi} I'(h)$. 

Observe that the second integral {in} 
\eqref{eq/ff2} is continuously differentiable with respect to
$h$ at least $[(n-1)/3]$ times. On the other hand, 
\begin{equation} \label{eq/decomposition_h}
\int_{-\varepsilon}^{\varepsilon} (x^2-h)^{(i-2)/3} \textup{d}x = c_i (-h)^{\frac{2i-1}{6}} + f_i(h), \quad h \le 0,
\end{equation}
where $f_i = f_i(h)$ is smooth and $c_i$ is a constant, which is zero for $i \equiv 2 \mod 3$ and non-zero for $i \not\equiv 2 \mod 3$; cf.\ \cite[Section 4]{Bolsinov2018}.
We infer that for $i = 3k+1$ and $i = 3k +3$, the integral is continuously differentiable exactly $k$ times. We shall now show that this implies 
$d_{3k} = 0$ and $d_{3k-2} = 0$ as long as $3k + 3 \le  n$. (Recall that $d_{2+3k}$ is zero by Lemma~\ref{lemma/2},
and so is the constant $c_{2+3k}$ by Eq.~\eqref{eq/decomposition_h}.) Proceeding inductively, we get that 
$$
d_{3k+3}\int_{-\varepsilon}^{\varepsilon} (x^2-h)^{k+1/3} \textup{d}x + d_{3k+1}\int_{-\varepsilon}^{\varepsilon} (x^2-h)^{k-1/3} \textup{d}x
$$
is continuously differentiable at least $k+1$ times for $3(k+2) \le n$. Using Eq.~\eqref{eq/decomposition_h}, we get that
$$
d_{3k+3} c_{3k+3} (-h)^{k+\frac{5}{6}} + d_{3k+1} c_{3k+1} (-h)^{k+\frac{1}{6}}, \quad h \le 0,
$$
must be $k+1$ times continuously differentiable for $3(k + 2) \le  n$. Since $c_{3k+3}$ and $c_{3k +1}$ are nonzero, this implies that $d_{3k+1} = 0$ and $d_{3k+3} = 0$ for 
$3(k + 2) \le n$. 

Coming back to Eq.~\eqref{eq/ff}, we infer that the first integral in the right hand side of this equation is, in fact, continuously differentiable at least $[(n-1)/3]-1$ times. But the same is true for the second integral. Since $n$ was arbitrary, we conclude that $u(x,y)$ is smooth.
\end{proof}

\begin{remark}[\textit{Normal form of a cusp singularity}]
  Observe that similarly to the hyperbolic case, we can transform a pair $(H = x^2-y^3, \omega)$ to the right normal form
  $$
  (H = x^2-y^3, \omega_{norm}), \quad \omega_{norm} = f(y) \textup{d}x \wedge \textup{d}y,
  $$
  where $f$ is a smooth function such that its Taylor coefficients $d_i$ vanish for $i \equiv 2 \mod 3$. In this case, the Taylor series of such a function $f$ is a complete invariant of the right symplectic equivalence.  This right normal form can be compared with the right normal form given in \cite{Bolsinov2018} in the analytic category.
  \end{remark}

\subsection{Global problem}

To conclude the section, we quote the well-known global results in the non-degenerate one degree of freedom case. Theorem~\ref{Theorem/symplectic_1dof} below can be obtained for instance using the results described above.  

Let $M^2$ be a compact $2$-dimensional surface (possibly with boundary) endowed with a volume form. Let $H$ be a smooth Morse function on this surface  with no critical points on the boundary $\partial M^2$. In particular, there are finitely many critical points of $H$. We shall moreover assume that $H$ is {\em simple}, that is, for each two distinct critical points $x_1, x_2 \in M^2$, we have $H(x_1) \ne H(x_2)$, and that the restriction of $H$ to the boundary $\partial M^2$ is locally constant.  

The function $H$ can be viewed as a Hamiltonian of a one degree of freedom Hamiltonian system on $M^2$. It gives rise to a singular fibration of $M^2$ into connected components of level sets $H^{-1}(h)$. Its regular fibers are circles. In other words, we have the (singular) fibration
$$\tilde{H} \colon M^2 \to B,$$ 
where $B$ is the Reeb graph (the bifurcation complex) of $H$. Note that this fibration is Lagrangian since the level sets $\tilde{H}^{-1}(h)$ are one-dimensional.

The symplectic classification (up to fiber-preserving symplectomorphisms) of such singular fibrations $\tilde{H} \colon M^2 \to B$
is well known; see, for instance, \cite{Vu-Ngoc2011, Dufour1994}. Basically, the only symplectic invariants of such fibrations are the action variables
$$
I(b) = \dfrac{1}{2\pi} \int\limits_{\tilde{H}^{-1}(b)} \alpha,
$$
where $\alpha$ is a primitive one-form for the symplectic structure (a one-form such that $\textup{d} \alpha$ is the symplectic form on $M^2$). More specifically, there is the following result.

\begin{theorem}[\textup{\cite{Vu-Ngoc2011, Dufour1994}}] \label{Theorem/symplectic_1dof}
Let $M^2_i$, $i = 1,2,$ be a symplectic compact surface and $H_i$ be a simple Morse function on $M^2_i$ whose restriction to $\partial M^2_i$ is locally constant. Consider the associated (singular) Lagrangian fibration $\tilde{H}_i \colon M^2_i \to B_i,$ where $B_i$ is the Reeb graph of $H_i$. Suppose
that there exists a diffeomorphism $\phi \colon B_1 \to B_2$ preserving action variables, that is, for every $b \in B_1$ and some choice\footnote{In a neighbourhood of every fiber $\tilde{H}^{-1}_i(b)$, a primitive one-form always exists.} of primitive one-forms $\alpha_1$ and $\alpha_2$ for the symplectic structures in neighbourhoods
of $\tilde{H_1}^{-1}(b)$ and $\tilde{H_2}^{-1}(\phi(b))$, we have  
$$I_1 = I_2 \circ \phi + c,$$ 
where
$I_i = \dfrac{1}{2\pi} \int\limits_{\tilde{H}^{-1}_i(\cdot)} \alpha_i$ and $c$ is a constant.
Then the diffeomorphism $\phi$ can be lifted to a fibration-preserving symplectomorphism $\Phi \colon M^2_1 \to M^2_2.$ 
\end{theorem} 

\begin{remark}
Here $\phi \colon B_1 \to B_2$ is a diffeomorphism  in the following sense:
$\phi$ can be lifted to a (not-necessarily canonical) diffeomorphism between the total
spaces $M^2_1$ and $M^2_2.$ 
\end{remark}

\begin{remark}
We note that the converse statement is obviously true; thus, the preservation of action variables is a necessary and sufficient statement for the existence of a fibration-preserving symplectomorphism $\Phi \colon M^2_1 \to M^2_2.$ 

We also note that a parametric version of Theorem~\ref{Theorem/symplectic_1dof} holds as well: if there is a pair of families of integrable systems smoothly depending on a parameter $\lambda$, then the existence of a smooth family of 
diffeomorphisms  $\phi_\lambda \colon B_{1, \lambda} \to B_{2, \lambda}$ implies the existence of a smooth family of symplectomorphisms $\Phi_\lambda \colon M^2_{1, \lambda} \to M^2_{2, \lambda}$ lifting $\phi_\lambda$. 
\end{remark}

\begin{remark}

We shall later apply this theorem to a special case of Morse functions having (at most) one minimum and one saddle point. Such a situation arises in the  integrable subcritical Hamiltonian Hopf bifurcation when one performs a symplectic reduction with respect to a global Hamiltonian circle action; it will be discussed in the Section~\ref{sec/HamHopf}.

\end{remark}

\section{Parabolic orbits in two degrees of freedom} \label{sec/cusp}

In this section, we shall prove a parametric version of Proposition~\ref{prop/cusp} and apply it to the local  symplectic classification of parabolic orbits in two
degrees of freedom systems; see Theorems~\ref{theorem_1} and \ref{theorem/parabolic_classification_c}. We prove these results in the smooth $C^\infty$ category (for the analytic case, see \cite{Bolsinov2018}). We will also give a symplectic normal form of a parabolic orbit in the analytic category.

Let $H = x^2 - y^3 + \lambda y$. We view this as a one degree  of freedom Hamiltonian, depending on the additional parameter $\lambda$. Let $\omega_{1,\lambda}$ and $\omega_{2,\lambda}$ be two symplectic forms in $\mathbb R^2$, smoothly depending on the parameter $\lambda$. Denote the difference $\omega_{1,\lambda} - \omega_{2,\lambda}$ by $g(x,y,\lambda) \textup{d}x\wedge \textup{d}y$. By the assumption, the function $g = g(x,y,\lambda)$ is smooth with respect to all of the variables $x,y$ and $\lambda$.

Observe that for $\lambda > 0$, the equation $H = h$ gives rise to a family of vanishing cycles projecting to the swallow-tail domain
\begin{equation*} 
D =  {\{(h,\lambda) \colon \lambda > 0,
h^2 < 4(\lambda/3)^3\}}. 
\end{equation*}
Consider the area functions (action variables) $I^\circ_1 = I^\circ_1(h, \lambda)$ and $I^\circ_2 = I^\circ_2(h, \lambda)$ corresponding to these vanishing cycles.

\begin{theorem} \label{theorem_1}
Suppose that the area functions $I^\circ_1(h,\lambda)$ and $I^\circ_2(h, \lambda)$, corresponding to the family of vanishing cycles on the swallow-tail domain $D$, coincide. Then $\omega_{1,\lambda}/\omega_{2,\lambda} > 0$ and there exists a family of diffeomorphisms $\Phi_\lambda$, smoothly depending on the parameter $\lambda$, such that
$H \circ \Phi_\lambda = H$ and $\Phi_\lambda^{*}(\omega_{2,\lambda}) = \omega_{1,\lambda}$.
\end{theorem}

First, we observe that in the case when $\omega_{1,\lambda} - \omega_{2,\lambda} = g(x,y,\lambda) \textup{d}x\wedge \textup{d}y$ is such that $g(x,y,\lambda)$ is an odd function of $x$, then the required diffeomorphism always exists. 

\begin{lemma} \label{lemma_2}
Let $\omega_{1,\lambda} - \omega_{2,\lambda} = g(x,y,\lambda) \textup{d}x\wedge \textup{d}y$ be such that $g(x,y,\lambda)$ is an odd function of $x$. Then $g(x,y,\lambda) = u_x H_y - u_y H_x$ for some smooth function $u = u(x,y,\lambda)$. In particular,  there exists a smooth family of diffeomorphisms $\Phi_\lambda$ such that
$H \circ \Phi_\lambda = H$ and $\Phi_\lambda^{*}(\omega_{2,\lambda}) = \omega_{1,\lambda}$.
\end{lemma}
\begin{proof}
Consider the equation $\textup{d}u \wedge \textup{d}H = \omega_{1,\lambda} - \omega_{2,\lambda}$, where $u = u(x,y,\lambda)$. It is equivalent to the equation
$$
-2x u_y + (\lambda - 3y^2) u_x = g(x,y,\lambda).
$$ 
Since $g(x,y,\lambda)$ is odd with respect to $x$, it can be written as $g(x,y,\lambda) = xg_1(x^2,y,\lambda)$, where $g_1$ is a smooth function.
It is left to observe that the function 
$$u = \tilde{u}(h,y,\lambda)\big|_{h = x^2-y^3+\lambda y} = \frac{-1}{2}\int_{\varepsilon}^y g_1(h+t^3-\lambda t,t,\lambda)\textup{d}t\big|_{h = x^2-y^3+\lambda y}$$
is a smooth solution. The existence of $\Phi_\lambda$ follows from Lemma~\ref{lemma/move}.
\end{proof}

Because of Lemma~\ref{lemma_2}, we can and will henceforth assume that 
$\omega_{1,\lambda} - \omega_{2,\lambda} = g(x^2,y,\lambda) dx \wedge dy$. Choosing the initial condition $u(0,y, \lambda) = 0$, we construct the required function $u$ for each value of the parameter $\lambda$ separately. For $\lambda = 0$, it is simply given by the formula (see Subsection~\ref{subsec/parabolic})
$$
u(x,y,0) = \int_0^x \dfrac{g(t^2, 
(t^2-h)^{1/3}, 0)}{-3(t^2-h)^{2/3}} \textup{d}t \big|_{h = H(x,y,0)}.
$$
For fixed $\lambda \ne 0$, we use the equality of the action variables $I^\circ_1(h,\lambda) = I^\circ_2(h, \lambda)$ for the family of vanishing cycles together with Propositions~\ref{prop/elliptic} and \ref{prop/hyperbolic}. We observe that the initial condition $u(0,y, \lambda) = 0$ defines $u$ unambiguously for all $\lambda$ and that Propositions~\ref{prop/elliptic} and \ref{prop/hyperbolic} imply $u$ is $C^{\infty}$ with respect to $(x,y)$ for all fixed $\lambda \ne 0$.

\begin{remark}In fact, the function $u$ is
$C^{\infty}$ with respect to all variables $(x,y, \lambda)$ outside the plane $\lambda = 0$ (see Remark~\ref{remark/parametric}), but we shall prove this independently. As we shall see, smoothness of $u$ on the plane $\lambda = 0$ with respect to $(x,y)$ is not needed for the proof. 
\end{remark}

Since $g$ is even with respect to $x$, the initial condition $u(0,y, \lambda) = 0$ implies $u = x a(x^2,y,\lambda)$, where the function $a$ is $C^{\infty}$ with respect to $(x^2,y)$ for fixed $\lambda \ne 0$.
Consider the Taylor expansion of the function $a$  with respect to the variable $x$:
$$
a = \sum^{N}_{k = 0} a_{k}(y,\lambda) x^{2k} + x^{2N+2} A(x^2,y,\lambda)
$$
and similarly for the function $g$:
$$
g = \sum^{N}_{k = 0} g_{k}(y,\lambda) x^{2k} + x^{2N+2} G(x^2,y,\lambda).
$$
Next we observe that $a$ solves the equation
$$
2x^2 \partial_y a + (3y^2-\lambda) (2x^2 \partial_{x^2} a + a) = g(x^2,y,\lambda).
$$
From this, we get that $(3y^2-\lambda) a_0 = g_0$ and for $k = 1, \ldots, N$,
$$
2 \partial_y a_{k-1} + (2k+3)(3y^2-\lambda) a_k = g_k.
$$
It follows (from Hadamard's lemma or Malgrange's preparation theorem) 
that all of the functions $a_k, k = 0, \ldots, N$ are smooth when $\lambda \ne 0$ and have a smooth extension to $\lambda = 0$. Hence, the problem reduces to the case
when
$$
g = x^{2N+2} (G(x^2,y,\lambda)-2a_N) = x^{2N+2} G_1(x^2,y,\lambda).
$$ 
But in this case, one can use the formula
\begin{equation} \label{formula/for_u} 
u = \textup{sign}(x) \big( \int_{f(x,y,\lambda)}^y  |h+t^3-\lambda t|^{N+\frac{1}{2}}G_1(h+t^3-\lambda t,t,\lambda)\textup{d}t\big)|_{h = H},
\end{equation}
where the function $f$ is such that the initial condition $u(0,y, \lambda) = 0$ is satisfied; specifically,
$f(x_0,y_0,\lambda)$ is defined as follows: starting from the point $(x_0,y_0) \in \mathbb R^2$, follow (the connected component of) the curve 
$$\{(x,y) \in \mathbb R^2 \colon H(x,y, \lambda) = H(x_0,y_0, \lambda)\}$$ 
in the $-y$ direction until the intersection with
the $y$-axis (so that, in particular, $H(0,f(x,y,\lambda),\lambda) = H(x,y,\lambda)$).

Observe that by the implicit function theorem, the function $f$ is smooth on the set $3f(x,y,\lambda)^2 - \lambda \ne 0.$ Hence, we can differentiate \eqref{formula/for_u} $N$ times on this set using the chain rule. Moreover, $H(0,f(x,y,\lambda),\lambda) = H(x,y,\lambda)$ implies we can ignore differentiation with respect to $f$ in the chain rule. For example,
\begin{multline*}
\partial_{\lambda} u =   \textup{sign}(x)\Big(-\partial_{\lambda} f \cdot 0 \ + \\
 \big(\int_{f(x,y,\lambda)}^y  \partial_\lambda \big(|h+t^3-\lambda t|^{N+\frac{1}{2}}G_1(h+t^3-\lambda t,t,\lambda)\big)\textup{d}t \big)\big|_{h = H} + \\
  \big(\int_{f(x,y,\lambda)}^y  \partial_h \big(|h+t^3-\lambda t|^{N+\frac{1}{2}}G_1(h+t^3-\lambda t,t,\lambda)\big)\textup{d}t\big)\big|_{h = H} H_\lambda \Big).
 \end{multline*}
This gives that
the partial derivatives of $u$ up to order $N$ exist when $3f(x,y,\lambda)^2 - \lambda \ne 0$ and are continuous on this set. 
We claim that all these partial derivatives extend continuously to the critical set  $3f(x,y,\lambda)^2 - \lambda = 0$. To show this, recall that derivatives of $f$ do not appear in the formula for partial derivatives of $u$ and that the integrand 
$$
|h+y^3-\lambda y|^{N+\frac{1}{2}}G_1(h+y^3-\lambda y,y,\lambda)
$$
and its derivatives up to order $N$ are everywhere continuous. Thus, the only discontinuity may arise from the discontinuity of the function $f$ on the set $3f(x,y,\lambda)^2 - \lambda = 0.$ Consider two cases: the case $\lambda \ne 0$ and the case $\lambda = 0$. For $\lambda \ne 0$, we can write
\begin{multline*}
u = \textup{sign}(x) \big( \int_{f(x,y,\lambda)}^{\sqrt{\lambda/3}}  |h+t^3-\lambda t|^{N+\frac{1}{2}}G_1(h+t^3-\lambda t,t,\lambda)\textup{d}t \big)|_{h = H} + \\
\textup{sign}(x) \big( \int_{\sqrt{\lambda/3}}^y  |h+t^3-\lambda t|^{N+\frac{1}{2}}G_1(h+t^3-\lambda t,t,\lambda)\textup{d}t \big)|_{h = H}.
 \end{multline*}
The equality of the action variables on the vanishing cycles implies that (the derivatives of) the first term in this sum tends to zero as we approach the critical set $3f(x,y,\lambda)^2 - \lambda = 0$. Furthermore, the terms appearing from the differentiation of $\sqrt{\lambda/3}$ will vanish in the limit, as  $\lambda$ is bounded away from zero and the integrand  
and its derivatives up to order $N$ are small near the origin. This proves the claim for $\lambda \ne 0$. 

Now assume $\lambda = 0$. In this case, it is sufficient to use formula \eqref{formula/for_u}: the jump of the function $f$ on the set $3f(x,y,\lambda)^2 - \lambda = 0$ tends to zero uniformly.
We conclude from that $u \in C^N$ is everywhere. Since $N$ can be made arbitrary large, Theorem~\ref{theorem_1} follows. \hfill $\square$

Theorem~\ref{theorem_1}  has an important corollary about local symplectic invariants of parabolic orbits in two
degrees of freedom systems. 
More specifically, consider a pair of two-degree of freedom integrable systems $
F_i \colon M_i \to \mathbb R^2, i = 1, 2,
$
with periodic orbit $\alpha_i$ of parabolic type. Let $V_i \simeq D^3 \times S^1$ be a sufficiently small neighbourhood of $\alpha_i$. By \cite{KudryavtsevaMartynchuk2021}, there exist  functions $H_i,J_i$ and coordinates $(x_i,y_i,\lambda_i, \varphi_i) \in D^3 \times S^1$
such that 

\begin{itemize}
\item[i)] $H_i = x_i^2 - y_i^3 + \lambda_i y_i$ and $J_i = \lambda_i$ are constant on the connected components of $F_i$, moreover $H_i$ is a function of $F_i$;

\item[ii)] $J_i$ is a smooth $2\pi$-periodic first integral. 

\end{itemize}

Finally, consider the action variable $I_i^\circ$ 
on the swallow-tail domain
$$
D_i = F_i({\{
\lambda_i > 0,
H_i^2 < 4(\lambda_i/3)^3, y_i<\sqrt{\lambda_i/3}\}}) \subset \mathbb R^2
$$
corresponding to the family of vanishing cycles 
$\{H_i=\mathrm{const}, \lambda_i=\mathrm{const}, \varphi_i=\mathrm{const}, y_i<\sqrt{\lambda_i/3}\}$
such that $I^\circ_i>0$ and $I^\circ_i(f)\to0$ as $f\to F_i(\alpha_i)$.
With this notation, we have the following result (this is in fact Theorem~\ref{theorem/parabolic_classification} in the introduction).

\begin{theorem} \label{theorem/parabolic_classification_c}
The fibrations given by $F_1$ and $F_2$ are symplectically equivalent near parabolic orbits $\alpha_1$ and $\alpha_2$ if and only if there exists a diffeomorphism germ 
$g \colon (\mathbb R^2,F_1(\alpha_1)) \to (\mathbb R^2,F_2(\alpha_2))$ that respects the swallow-tail domains, $g(D_1)=D_2$, and makes the actions equal, 
$I^\circ_1 = I^\circ_2 \circ g  \mbox{ and } J_1 = J_2 \circ g$ on $D_1$.
\end{theorem}
\begin{proof}
The proof follows from Theorem~\ref{theorem_1}; cf.\ \cite{Bolsinov2018} and, in particular,  
\cite[Proposition 3.3]{Bolsinov2018}.
\end{proof}
\begin{remark}
Observe that 
this theorem implies that the action variables $I^\circ$ and $J$ in the swallow-tail domain form a complete set of symplectic invariants of a parabolic orbit in the local case also in the smooth category; cf.\ the analytic case \cite{Bolsinov2018}.
\end{remark}

Finally, note that Theorem~\ref{theorem/parabolic_classification_c} can be strengthened in the following way
(cf.\ Remark~\ref{remark/affine}). Observe that  $I^\circ_i$ and $J_i$
define (germs of) integer affine structures on the closures of the corresponding swallow-tail domains $D_i$. 
Then the existence of a diffeomorphism $\phi$ preserving these integer affine structures is a
necessary and sufficient condition for the existence of a fiberwise symplectomorphism $\Phi$ between saturated 
neighbourhoods of parabolic orbits $\alpha_1$ and $\alpha_2$; moreover,  such a symplectomorphism $\Phi$ can always be chosen as a lift of $\phi$. This shows that the integer affine structure is a complete symplectic invariant for
parabolic orbits also in the smooth category; cf.\ the analytic case \cite{Bolsinov2018}. 

\begin{remark} \label{rem:affine}
Using this result, it can be shown that the same is true in the 
semi-local setting, that is, in a saturated neighbourhood of a cuspidal torus. More specifically, a complete symplectic invariant for cuspidal tori (in the smooth category) is given by (germs of) integer affine structures on the 
both open strata of the (local) bifurcation complex, which are defined by the action variables $(I^\circ, J)$ on the swallow-tail domain and by action variables $(I, J)$ on the other open stratum.\footnote{By replacing $I$ with $\pm I+\mathrm{const}$, we can and will assume that $\partial_I H(F(I,J))>0$ and $I(f)\to0$ as $f\to F(\alpha)$, where $\alpha$ is the parabolic orbit and $H=H(F)$ has the form $H=x^2-y^3+Jy$ near $\alpha$.}
The $(I^\circ, J)$-image of the swallow-tail domain and the $(I, J)$-image of the other open stratum are shown in Fig.~\ref{fig/affine}, left (top and bottom, resp.), and have the form $\{(I^\circ,\lambda)\mid 0<I^\circ<S_2(\lambda)-S_1(\lambda),\ \lambda>0\}$ and $\mathbb R^2\setminus\{(I,\lambda)\mid S_1(\lambda)\le I\le S_2(\lambda),\ \lambda\ge0\}$, resp., 
for some continuous functions $S_1(\lambda)$ and $S_2(\lambda)$ such that $S_i(0)=0$ and $S_1(\lambda)<S_2(\lambda)$ for $\lambda>0$.
One can show (using Theorem \ref {theorem_analytic_right_left}) that, in the real-analytic case, $S_i(\lambda)$ are $C^1$-smooth and $S_i'(\lambda)=a+a_i\lambda^{1/4}+O(\lambda^{1/2})$, $\lambda\ge0$, where $0 < a_1 < a_2$. Without loss of generality, $0\le a<1$ (this can be achieved by replacing $I$ with $I+kJ$, $k\in\mathbb Z$).
Such a pair of function germs $S_1(\lambda)$ and $S_2(\lambda)$, $\lambda\ge0$, at $0$ is a symplectic invariant of cuspidal tori, since it is uniquely determined by the integer affine structures on open strata of the bifurcation complex. Nevertheless, this symplectic invariant is incomplete (cf.\ Theorems \ref {theorem/parabolic_classification_c} and \ref {theorem_analytic_right_left} for complete symplectic invariants).
\end{remark}

\begin{figure}[htbp]
\begin{center}
\includegraphics[width=0.98\linewidth]{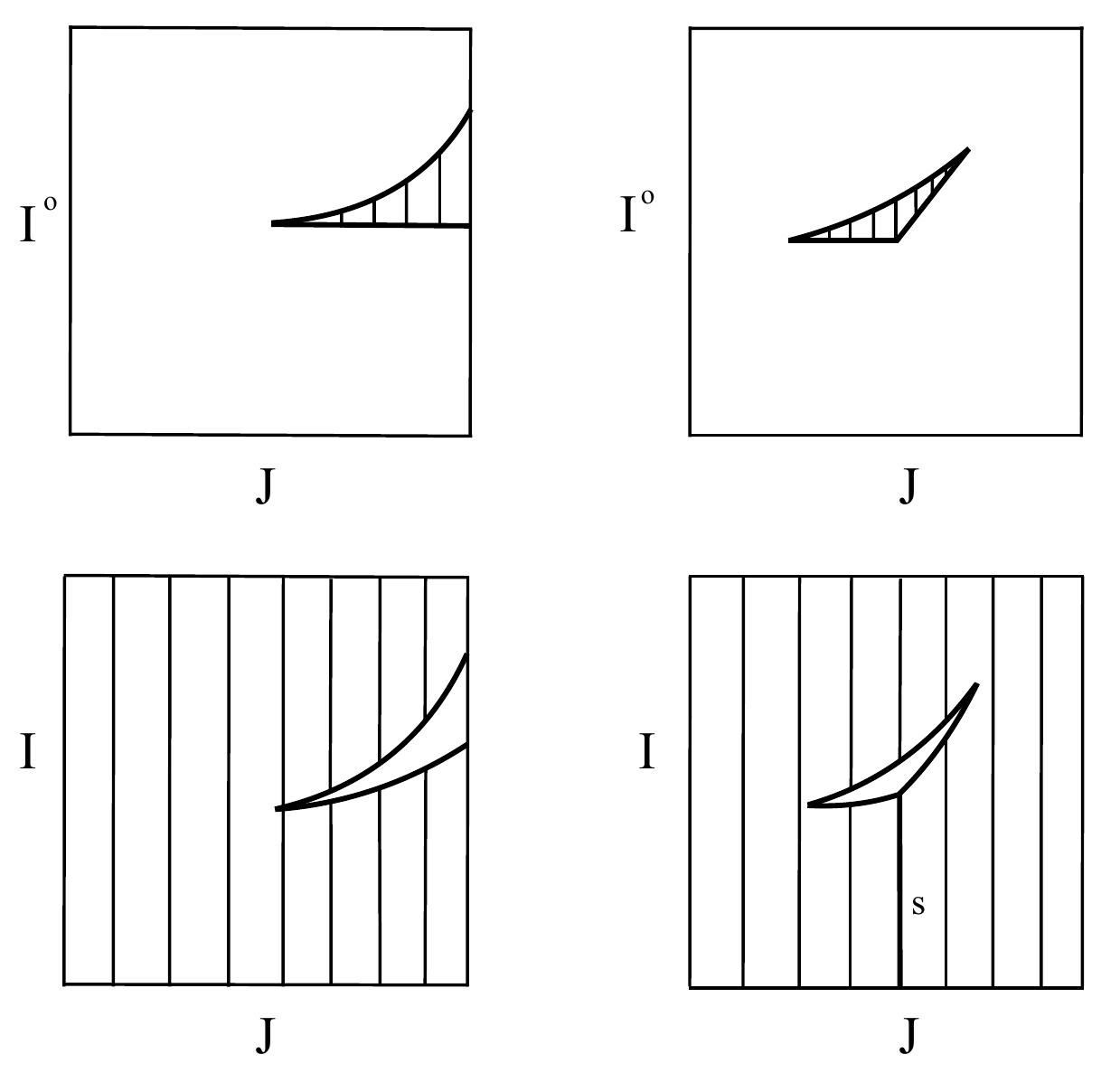}
\caption{
The integer affine structures on open strata of the local bifurcation complex at a cuspidal torus (left), at a flap (right). 
The $(I^\circ,J)$-image of the swallow-tail domain (top left), and the $(I,J)$-image of the other open stratum (bottom left) near a cuspidal torus. 
The $(I^\circ,J)$-image of the flap (top right), and the $(I,J)$-image of the flap base cut along $s$ (bottom right) near a flap.}
\label{fig/affine}
\end{center}
\end{figure}

\subsection{Normal forms, analytic case} \label{normal_forms}

In this section, we derive two normal forms for a parabolic orbit in the analytic case, up to 
the right and right-left equivalences, respectively. The first normal form is based
on the proof of Theorem~\ref{theorem_1}. The second normal form uses a different
approach, based on a canonical normal form in a neighbourhood of a parabolic point 
(see \cite{VarchenkoGivental1982} and Theorem~\ref{theorem_parabolic_point} below).

Let us start with the case of the right equivalence. Consider a real-analytic integrable system 
$F \colon M \to \mathbb R^2$
with a parabolic orbit $\alpha$. By \cite{Zung2000, Kudryavtseva2020, Bolsinov2018}, there exist  
real-analytic functions $H = H(F_1, F_2)$ and $J = J(F_1, F_2)$ in 
a neighborhood $V \simeq D^3 \times S^1$ of $\alpha$
and (real-analytic) coordinates $(x,y,\lambda, \varphi)
\in D^3 \times S^1$
such that 

\begin{itemize}
\item[i)] $H = x^2 - y^3 + \lambda y$ and $J = \lambda$;

\item[ii)] $J$ is an analytic $2\pi$-periodic integral.

\end{itemize}

\begin{theorem} \label{theorem_parabolic_normal_form}
Consider a parabolic orbit $\alpha$ of a real-analytic integrable system 
$F = (F_1,F_2) \colon M \to \mathbb R^2.$ Let the functions $H, J$ and coordinates 
$(x,y,J = \lambda, \varphi)$ be as above. 
Then, up to the real-analytic right symplectic equivalence, the functions $H$, $J$ and the symplectic structure $\omega$ have the following local normal form
near $\alpha$:
$$
 H = x^2 - y^3 + \lambda y, \quad J = \lambda, \qquad 
\omega_{norm} = c(x^2,y) \textup{d}x \wedge \textup{d}y + \textup{d} J \wedge \textup{d} \varphi,
$$
for some uniquely defined real-analytic function germ $c = c(x^2,y)$ with $c(0,0) > 0.$ Moreover,  if $\partial_{F_2}J(F(\alpha))\ne0$ and the function germ $c = c(x^2,y)$ corresponds to the (uniquely defined) function $H$ having the form $H = \frac{\pm F_1-a(J)}{b^{3/2}(J)}$ (see \cite[Sec.~2]{Bolsinov2018} for an explicit construction of such $H$), then the local bifurcation diagram of $F$ at $\alpha$ and the germs $J = J(F_1,F_2)$ and $c = c(x^2,y)$ (at $F(\alpha)$ and $(0,0)$, resp.)
form a complete  set of invariants of the real-analytic right symplectic equivalence.
\end{theorem} 

This theorem follows from the following proposition ({\'a} la \cite[Lemme principal]{DeVerdiere1979}),
the proof of which is given below.

\begin{proposition} \label{proposition_parabolic_normal_form}
Let $\alpha$ be a parabolic orbit of a real-analytic integrable system. Then there exist coordinates
$(x,y,J = \lambda, \varphi) \in D^3 \times S^1$ such that $(x,y, J = \lambda)$ 
satisfy {\rm i)--ii)} from above and the (real-analytic) symplectic form $\omega$ is given by
$$
\omega = c(x^2,y) \textup{d}x \wedge \textup{d}y + \textup{d}u \wedge \textup{d}H + \textup{d} J \wedge \textup{d} \varphi
$$
for some real-analytic function germs $c = c(x^2,y), \ c(0,0) >
0,$ and $u = u(x,y,J)$. The function
$c$ is uniquely defined.
\end{proposition}

We first prove the following lemma.

\begin{lemma} \label{lemma_convergence}
Let
$$
g = \sum^{\infty}_{k = 0} g_{k}(y,\lambda) x^{2k}
$$ be a germ of a real-analytic function. Define germs of real-analytic functions $b_{k} = b_k(y,\lambda)$ and 
$c_k = c_k(y), \ k \ge 0,$  inductively as follows: 
\begin{align} \label{eq/b_and_c}
g_{0} &= (3y^2-\lambda) b_{0}(y,\lambda) + c_{0}(y) \quad \mbox{ and }  \\ 
g_{k} -  2\partial_y b_{k-1} &= (2k+3)(3y^2-\lambda) b_{k}(y,\lambda) + c_{k}(y). \label{eq/b_and_c_2}
\end{align}
Then the series
$$
b = \sum^{\infty}_{k = 0} b_{k}(y,\lambda) x^{2k} \ \mbox{ and } \
c = \sum^{\infty}_{k = 0} c_{k}(y) x^{2k}
$$
converge to real-analytic function germs.
\end{lemma}
\begin{remark}
The existence and uniqueness of  analytic function $b_{k} = b_k(y,\lambda)$ and 
$c_k = c_k(y)$ satisfying Eq.~(\ref{eq/b_and_c}-\ref{eq/b_and_c_2}) follows from Adamar's lemma.
As we shall see, when $g$ is coming from a symplectic form, the 
function $c$ is the  symplectic invariant appearing in Theorem~\ref{theorem_parabolic_normal_form} above 
(cf.\ the proof of 
Theorem~\ref{theorem_1}, where the same equations as Eq.~(\ref{eq/b_and_c}-\ref{eq/b_and_c_2}) appear with all $c_k$ equal to zero).

We also note that the function $b$ is real-analytic on a possibly smaller neighbourhood of the origin than the domain
of analyticity of $g$.
\end{remark}

\renewcommand*{\proofname}{Proof of Lemma~\ref{lemma_convergence}}

\begin{proof}
Let
$$
g = \sum^{\infty}_{i,j,k} c_{i,j,2k} y^i \lambda^j x^{2k}.
$$
This series converges since $g$ is real-analytic. Consider also 
the
(at the moment formal) series 
$$b = \sum^{\infty}_{i,j,k} b_{i,j,2k} y^i \lambda^j x^{2k}.$$
Then for $u = x b(x^2,y,\lambda)$, we have that $g - (u_x H_y - u_y H_x)$ is a formal series in $x^2$ and $y$ only
(cf.\ the proof of Theorem~\ref{theorem_1}).

Consider the function
$v = x^{2k+1}y^i \lambda^{j-1}$. Then
$$
v_x H_y - v_y H_x = 
(3 (2k+1)3 x^{2k} y^{i+2} - 2i x^{2k+2} y^{i-1})\lambda^{j-1} - (2k+1)x^{2k} y^i \lambda^j.
$$
Since $g - (u_x H_y - u_y H_x)$ does not depend on $\lambda$, it follows that for all $i,j,k$ with $j \ge 1$, we have
\begin{equation*}
c_{i,j,2k} = 3(2k+1) b_{i-2,j,2k} - (2k+1)b_{i,j-1,2k} - 2(i+1)b_{i+1,j,2(k-1)}.
\end{equation*}
These equations determine $b_{ij2k}$ unambiguously. Let us rewrite them
in the following form:
\begin{align} \label{eq/def_equations_b}
b_{i,j,2k} &= 3 b_{i-2,j+1,2k} - \frac{2(i+1)}{2k+1} b_{i+1,j+1,2(k-1)} - \frac{c_{i,j+1,2k}}{2k+1}. 
\end{align}

Now recall that $g$ is real-analytic. Hence
there exists radius $r > 3$ such that $|c_{i,j,2k}| \le r^{i+j+2k}$. Using Eq.~\eqref{eq/def_equations_b},
induction by $k$ shows that for some constant $c,$ we have
$$
|b_{i,j,2k}| \le \frac{(i+k)^k}{k!} r^{i+j+2k} c^k.
$$
It now follows that $|b_{i,j,2k}| \le R^{i+j+2k}$
for some $R > \max\{r,c\}$ (one can take $R = (\max\{r,c\})^3$). 

We conclude that $b$ is represented by a convergent power series and
is therefore (real-)analytic. It follows that
$\sum^{\infty}_{k = 0} b_{k}(y,\lambda) x^{2k}
$
also converges and to the same function $b$. From Eq.~(\ref{eq/b_and_c}-\ref{eq/b_and_c_2}), it now follows that
$
\sum^{\infty}_{k = 0} c_{k}(y) x^{2k}
$
converges as well.
\end{proof}

\renewcommand*{\proofname}{Proof of Proposition~\ref{proposition_parabolic_normal_form}}

\begin{proof}
First observe that, after taking a smaller neighbourhood $V$ of the parabolic orbit and changing $x,y,\varphi$ if necessary, the symplectic form can be written as
\begin{equation} \label{eq/symplectic_structure}
\omega = \tilde{g}(x,y,\lambda)\textup{d}x\wedge \textup{d}y + \textup{d}\lambda\wedge(\textup{d}\varphi+\tilde A(x,y,\lambda)\textup{d}x + \tilde B(x,y,\lambda)\textup{d}y),
\end{equation}
for some real-analytic functions $\tilde{g}, \tilde A$ and $\tilde B$; see \cite{Bolsinov2018}.
By (a real-analytic version of) Lemma~\ref{lemma_2}, 
we can write 
\begin{equation*} 
\omega = g(x^2,y,\lambda)\textup{d}x\wedge \textup{d}y + \textup{d}v_1 \wedge \textup{d}H + \textup{d}\lambda\wedge(\textup{d}\varphi+A_1(x,y,\lambda)\textup{d}x + B_1(x,y,\lambda)\textup{d}y),
\end{equation*}
where $g = g(x^2,y,\lambda)$ is the even part of $\tilde{g}(x,y,\lambda)$ with respect to $x$, and $v_1 = v_1(x,y,\lambda)$, $A_1= A_1(x,y,\lambda)$ and $B_1 = B_1(x,y,\lambda)$ are some real-analytic functions. 
Consider
the Taylor expansion
$$
g = \sum^{\infty}_{k = 0} g_{k}(y,\lambda) x^{2k}. 
$$
By the previous lemma, there exist analytic functions 
$$
b = \sum^{\infty}_{k = 0} b_{k}(y,\lambda) x^{2k} \ \mbox{ and } \
c = \sum^{\infty}_{k = 0} c_{k}(y) x^{2k}
$$
such that 
$$g_{0} = (3y^2-\lambda) b_{0}(y,\lambda) + c_{0}(y).$$
and 
$$g_{n} -  2\partial_y b_{n-1}= (2n+3)(3y^2-\lambda) b_{n}(y,\lambda) + c_{n}(y).$$
Let $v_2 = xb(x^2,y,\lambda)$. Then, by the construction (cf.\ the proof of Theorem~\ref{theorem_1}), the symplectic $v_2$-move transforms $g(x^2,y,\lambda) \textup{d}x \wedge \textup{d}y$ to 
$c(x^2, y) \textup{d}x \wedge \textup{d}y$. Hence we can write
$$
g(x^2,y,\lambda) \textup{d}x \wedge \textup{d}y = c(x^2,y) \textup{d}x \wedge \textup{d}y + (\partial_x(v_2) \partial_y(H) - \partial_y(v_2) \partial_x(H)) \textup{d}x \wedge \textup{d}y
$$
and 
$$
\omega = c(x^2,y) \textup{d}x \wedge \textup{d}y + \textup{d}u \wedge \textup{d}H + \textup{d}\lambda\wedge(\textup{d}\varphi+ A_2(x,y,\lambda)\textup{d}x + B_2(x,y,\lambda)\textup{d}y),
$$
where $u = v_1 + v_2$.
Since $\omega - c(x^2,y) \textup{d}x \wedge \textup{d}y - \textup{d}u \wedge \textup{d}H$ is closed, we get the formula
$$
\omega = c(x^2,y) \textup{d}x \wedge \textup{d}y + \textup{d}u \wedge \textup{d}H + \textup{d} J \wedge \textup{d}(\varphi+w)
$$
for some real-analytic germ $w = w(x,y,\lambda)$ at the origin. To put this into the required form, it is left 
to change the angle coordinate $\varphi \mapsto \varphi + w$. 
Next we observe that because of Theorem~\ref{theorem_1} and Eq.~(\ref{eq/b_and_c}--\ref{eq/b_and_c_2}), the representative $c = c(x^2,y)$ is unique, provided that the coordinates $(x,y,\varphi)$ satisfying \eqref {eq/symplectic_structure} are fixed.
Finally, observe that replacing $x$ with $-x$ gives rise to replacing $\tilde g(x,y,\lambda)$ with $-\tilde g(-x,y,\lambda)$ in \eqref{eq/symplectic_structure},
and hence $c(x^2,y)$ with $-c(x^2,y)$.
Thus, by replacing $x$ with $-x$ if necessarily, we can achieve that $c(0,0)>0$.
\end{proof}

\renewcommand*{\proofname}{Proof of Theorem~\ref{theorem_parabolic_normal_form}}

\begin{proof}
First we observe that the pair of functions $H=H(F_1,F_2), J=J(F_1,F_2)$ satisfying i)--ii) is non-unique and is defined up to replacing it with $H+K, J$, where $K = K(F_1, F_2)$ is any real-analytic germ at $(0,0)$ vanishing on the curve $\Sigma=\{(F_1,F_2)\colon H^2(F_1,F_2) = 4(J(F_1,F_2)/3)^3\}\subset\mathbb R^2$, which is the local bifurcation diagram of $F$ at $\alpha$. One can show that the germ $c = c(x^2,y)$ depends on the choice of such $H$, so it is not uniquely determined by $F=(F_1,F_2)$.
However, if $\partial_{F_2}J(F(\alpha))\ne0$, then after the change $(F_1, F_2)\to(F_1, J)$, we can choose $H$ uniquely by requiring that it has the form 
$H=\frac{\pm F_1-a(J)}{b^{3/2}(J)}$ for some real-analytic germs $a(J)$ and $b(J)$ at $0$, $b(0)>0$ (see \cite[Sec.~2]{Bolsinov2018} for an explicit construction of such $H$).

Next we conclude from Proposition~\ref {proposition_parabolic_normal_form} and Lemma \ref {lemma/move} that the bifurcation diagram $\Sigma$, the action variable $J =J(F_1,F_2)$ and the function $c = c(x^2,y)$ with $c(0,0)>0$ are complete real-analytic right symplectic invariants.
Moreover, any curve $\Sigma$, which is a diffeomorphic image of the semicubical parabola with a vertex at $F(\alpha)$, and any analytic functions $J=J(F_1,F_2)$ and $c = c(x^2,y)
$ appear as real-analytic right symplectic invariants for some integrable system with a parabolic orbit, provided that $\partial_{F_2}J(F(\alpha))\ne0$, $\xi J\ne0$ and $c(0,0) > 0$, where $\xi$ is a non-zero vector at $F(\alpha)$ tangent to $\Sigma$.
\end{proof}

\renewcommand*{\proofname}{Proof}

To obtain a
right-left normal form, we will use a different approach. First observe that by
\cite{VarchenkoGivental1982}, all parabolic points are symplectically equivalent. More specifically,
there is the following result.
\begin{theorem}[\cite{VarchenkoGivental1982}] \label{theorem_parabolic_point}
Let $P$ be a parabolic point of a real-analytic integrable system $F  = (F_1, F_2)\colon U \to \mathbb R^2$. Then there exist
coordinates $(\tilde x,\tilde y, \tilde \lambda, \tilde \mu)$ centered at $P$ such that
$$\tilde H = \tilde{x}^2 - \tilde{y}^3 + \tilde{\lambda} \tilde{y} \mbox{ and } \tilde{J} = \tilde \lambda$$
are  real-analytic functions of $(F_1,F_2)$ and the symplectic structure has the canonical form
$$
\textup{d} \tilde{x} \wedge \textup{d} \tilde{y} + \textup{d} \tilde{\lambda} \wedge \textup{d} \tilde \mu.
$$
 The functions $\tilde H$ and $\tilde{J}$ are uniquely defined.
\end{theorem} 
\begin{proof}
The result follows from \cite[Theorem 3]{VarchenkoGivental1982}; see also \cite{Garay2004}.
Note that a parabolic singularity is infinitesimally stable (see \cite[Theorem 5.25]{Zoladek2006} for a proof).
\end{proof}

\begin{theorem} \label{theorem_analytic_right_left}
Let $\alpha$ be a parabolic orbit of a real-analytic integrable system $F \colon U \to \mathbb R^2$. 
Consider canonical coordinates $(\tilde x,\tilde y, \tilde J = \tilde \lambda, \tilde \mu)$ in a neighbourhood of 
some point $P \in \alpha$ as in Theorem~\ref{theorem_parabolic_point}.
Let $J$ be the $2 \pi$-periodic integral of the system in a neighbourhood of
$\alpha$ such that 
$$J(0,0) = 0 \ \mbox{ and } \ \partial_{\tilde J}J
(0,0) > 0.$$ Then (the germ at $(0,0)$ of)
the real-analytic function $J = J(\tilde H, \tilde J)$ 
classifies the singularity up to
the real-analytic right-left symplectic equivalence. 
Moreover, 
the functions $\tilde H,\tilde J$ and the symplectic structure $\omega$ have the following local normal form near $\alpha$:
$$
\tilde H = u^2 - v^3 + \tilde \lambda v, \quad \tilde J = \tilde \lambda, \qquad
\omega = \textup{d}u \wedge \textup{d}v + \textup{d} J(\tilde H,\tilde J) \wedge \textup{d} \psi
$$
in some real-analytic coordinates $(u,v,\tilde J=\tilde\lambda, \psi) \in D^3 \times S^1$ on a neighbourhood of $\alpha$, with $\alpha=(0,0,0)\times S^1$.
\end{theorem} 

\begin{proof}
The classification assertion follows directly from Theorem \ref {theorem_parabolic_point} and the existence of a $2\pi$-periodic first integral $J$ in a neighbourhood of a parabolic orbit, as we now show. 

First observe that since a $2\pi$-periodic integral $J$
is an action variable of the system and since
the conditions $J(0,0) = 0$ and $\partial_{\tilde J}J
(0,0) > 0$ determine 
$J = J(\tilde H, \tilde J)$ unambiguously, the germ of $J = J(\tilde H, \tilde J)$ at $(0,0)$ 
is indeed a symplectic invariant. 

To show that a real-analytic germ $J = J(\tilde H, \tilde J)$ at $(0,0)$  is the only  invariant, and it admits arbitrary values,
we now construct a `model' real-analytic integrable system with a parabolic orbit from such a germ.
Consider $\mathbb R^4$ with standard canonical coordinates
$(\tilde x,\tilde y, \tilde J = \tilde \lambda, \tilde \mu)$. Let 
$$\tilde H = \tilde{x}^2 - \tilde{y}^3 + \tilde{\lambda} \tilde{y} \mbox{ and } \tilde{J} = \tilde \lambda$$
and the function $T = T(\tilde x,\tilde y, \tilde \lambda)$ be defined by $T =  2\pi \partial_{\tilde J}J(\tilde H, \tilde J).$
Take a small disk $D^3$ in the $(\tilde x,\tilde y, \tilde \lambda)$-space around the origin and the cylinder $C$ over this
disk of height $T$:
$$
C = \{ (\tilde x,\tilde y, \tilde \lambda, \tilde \mu) \in \mathbb R^4 \mid 0 \le \tilde\mu \le T(\tilde x,\tilde y, \tilde \lambda) \}.
$$ 
Then the phase space $U$ of the desired integrable system is given as the quotient space
$$
C / \sim, \quad (\tilde x,\tilde y, \tilde\lambda,
0) \sim \Phi(\tilde x,\tilde y, \tilde \lambda, 
0),
$$
where $\Phi = \Phi(\tilde x,\tilde y, \tilde \lambda
)$ is the time-$2\pi$ map of the Hamiltonian flow of
$J$ on $C$. The quotient space naturally inherits a smooth structure, a symplectic structure $\omega / \sim$ (from the canonical symplectic
structure $\omega$ on $C$) and the functions $\tilde H$ and $\tilde{J}$ in involution (this follows by performing a general procedure of Hamiltonian gluing introduced by Lazutkin \cite[Chap.~I, Secs.~4 and~6]{Lazutkin1993}; this procedure can be applied in our case, since the Hamiltonian vector field generated by $J$ is transversal to $D^3$ and $\Phi|_{D^3}$ preserves $\omega|_{D^3}$).  By the construction,
$$(C / \sim, F, \omega / \sim), \quad  F = (\tilde H, \tilde{J}),$$
is an integrable system with a parabolic orbit $\tilde x = \tilde y = \lambda = 0$. Moreover, every real-analytic parabolic 
singularity
is locally right-left equivalent to such a model:  it suffices to cut out an embedded disk $D^3$ transversal to the parabolic orbit, given by $\tilde \mu = 0$ in
canonical coordinates as in Theorem~\ref{theorem_parabolic_point}, to obtain a cylinder $C$ as above.

To prove the local normal form, it suffices to define new coordinates $u,v,\psi$ as follows. Consider coordinates
$\tilde x, \tilde y, \tilde \lambda, \tilde \mu$ as in Theorem~\ref{theorem_parabolic_point}.
On the embedded disk $D^3$ given by $\tilde\mu = 0$, the new coordinates $u,v$ simply coincide with $\tilde x, \tilde y$. These coordinates $u,v$ on $D^3$ are then extended to a neighbourhood 
of the parabolic orbit $\alpha$ using the flow of the $2\pi$-periodic integral $J$, by making them invariant under this flow. The remaining angle coordinate $\psi$
is simply the time needed to reach a given point from $D^3$ using the flow of $J$.
\end{proof}

\renewcommand*{\proofname}{Proof}

\section{Symplectic invariants of integrable systems arising via an integrable Hamiltonian Hopf bifurcation} \label{sec/HamHopf}

The goal of this section is to show that for a certain range of the bifurcation parameter, the only symplectic invariants of the integrable Hamiltonian Hopf bifurcation are the action variables.  

Assume that we are given an integrable 2 degree of freedom Hamiltonian system
$$F = (H,J) \colon M \to \mathbb R^2$$
with (only) the following singularities

\begin{itemize}
\item[i)] two families of elliptic orbits joining at an elliptic-elliptic point,

\item[ii)] a family of hyperbolic orbits with an oriented separatrix diagram joining the two elliptic families at two parabolic orbits.
\end{itemize}

We shall moreover assume that 

\begin{itemize}
\item[iii)] the fibers of $F$ are compact,
\end{itemize}
and, by changing the functions $H$ and $J$ if necessary, that

\begin{itemize}
\item[iv)] $J$ generates a global Hamiltonian circle action that is free outside the elliptic-elliptic point.
\end{itemize}

We note that the assumptions i)--iii) guarantee  the existence of such a Hamiltonian action: this is well-known on a neighbourhood of the family of hyperbolic fibres, and follows from \cite{KudryavtsevaMartynchuk2021} on a neighbourhood of the closure of this family, see \cite[Proposition 5]{Efstathiou2012} for the real-analytic case.

Under these assumptions, the bifurcation diagram is topologically a triangle in $\mathbb R^2$; see Figure~\ref{figure/qsp}. The closure of the region
enclosed by this triangle is usually referred to as a \textit{flap} \cite{Efstathiou2012}. Every value of $F$ that is in the interior of the flap lifts to a union of two tori and every value in the exterior lifts  to a single torus. This is a typical situation in the case of the 
integrable subcritical Hamiltonian Hopf bifurcation \cite{vanderMeer1985} 
(see also \cite{DullinPelayo2016, Hohloch2021}).

\begin{lemma} \label{lemma1}
Consider two integrable Hamiltonian systems 
$$F_1 \colon M_1 \to \mathbb R^2 \ \mbox{ and } \ F_2 \colon M_2 \to \mathbb R^2$$
satisfying the above assumptions {\rm i)--iv)}
and the associated (singular) Lagrangian torus fibration $\mathcal F_i
\colon M^2_i \to B_i,$ where $B_i$ is the bifurcation complex of $F_i.$
Suppose that there exists a diffeomorphism\footnote{In the sense of footnote 3.}
$$\phi \colon B_1 \to B_2$$ preserving the integer affine structures.
Furthermore, suppose
that $\phi$ lifts to a fibration-preserving symplectomorphism of neighbourhoods  of the fibers containing the parabolic orbits (the cuspidal tori).  
Then $\phi$ can be lifted to a fibration-preserving symplectomorphism $\Phi \colon M^2_1 \to M^2_2.$
\end{lemma}

\begin{proof}
Let us denote a lift of $\phi$ to a neighbourhood $U$  of the cuspidal tori by $\Phi_1$. Such a lift exists by the assumption. Using the parametric version of   Theorem~\ref{Theorem/symplectic_1dof}, after possibly changing the symplectomorphism  $\Phi_1$ near the boundary of $U$, it can be extended to a symplectomorphism $\Phi_2 \colon M^2_1 \setminus V_1 \to M^2_2 \setminus V_2$, where $V_1$ and $V_2$ are arbitrary small balls containing the elliptic-elliptic points of the two systems.

By the Eliasson theorem, there exists a fibration preserving symplectomorphism $\Psi$ between neighbourhoods $W_1$ and $W_2$ of the elliptic-elliptic points. Without loss of generality, this symplectomorphism is a lift of $\phi$ and the neighbourhood $W_1$ contains the closure $\overline{V_1}$ of $V_1$. On the set $W_1 \setminus V_1$ we therefore get a symplectomorphism $\Psi^{-1} \circ \Phi_2$ preserving $F_1$. It can be shown that this symplectomorphism arises as the time one map for the Hamiltonian flow of a `generating' function $S \colon F_1(W_1 \setminus V_1) \to \mathbb R.$ Replacing 
$S$ with a new function $\chi S$, where $\chi$ is a smooth bump function that is $0$ on $F(\partial W_1)$ and $1$ on $F(\partial V_1)$, gives rise to a fibration-preserving symplectomorphism defined on $W_1 \setminus V_1$ and interpolating between $\Psi$ and $\Phi_2$.
\end{proof}

It is known that in the analytic case, the only semi-local invariants of cuspidal tori are the action variables \cite{Bolsinov2018}. But Theorem~\ref{theorem/parabolic_classification_c} implies that the same is true in the $C^{\infty}$ category.
Combining this with 
Lemma~\ref{lemma1}, we arrive at the following main result.

\begin{theorem} \label{theorem/main2}
Consider two integrable Hamiltonian systems 
$$F_1 \colon M_1 \to \mathbb R^2 \ \mbox{ and } \ F_2 \colon M_2 \to \mathbb R^2$$
satisfying the above assumptions {\rm i)--iv)}
and the associated (singular) Lagrangian torus fibrations $\mathcal F_i \colon M^2_i \to B_i,$ where $B_i$ is the bifurcation complex of $F_i$. Suppose
that there exists a diffeomorphism 
$$\phi \colon B_1 \to B_2$$ preserving the integer affine structures. 
Then $\phi$ can be lifted to a fibration-preserving symplectomorphism $\Phi \colon M^2_1 \to M^2_2.$
\hfill $\square$
\end{theorem}

\begin{remark}
We note that a similar result holds in the local case of integrable subcritical Hamiltonian Hopf bifurcation \cite{vanderMeer1985}, regarded as a 3 degree of freedom integrable system in a neighbourhood of a rank one singular orbit. More specifically, fixing a small positive bifurcation parameter, one gets the same bifurcation diagram as in 
Theorem~\ref{theorem/main2}, but now $F_i^{-1}(f)$ is a union 
$T^2 \sqcup C^2$ of a 2-torus and a 2-cylinder if $f$ is in the interior of the flap and 
a 2-cylinder $C^2$ if $f$ is in the complement of the flap. Also in this case, every diffeomorphism 
$\phi \colon B_1 \to B_2$ preserving the integer affine structures in the flaps
can be lifted to a fibration-preserving symplectomorphism. 
\end{remark}

Thus, by Theorem~\ref {theorem/main2}, a complete symplectic invariant on a small neighbourhood of a flap is given by (germs of) integer affine structures on 
two open strata of the bifurcation complex. These integer affine structures are defined by the action variables $(I^\circ, J)$ on the flap (and on the corresponding stratum of the bifurcation complex) and by (multi-valued) action variables $(I, J)$ on the other open stratum called the {\em flap base} \cite{Efstathiou2012}.
The $(I^\circ, J)$-image of the flap is shown in Fig.~\ref{fig/affine}, top right, and has the form 
$$
\{(I^\circ,\lambda)\mid \max\{0,\lambda\} \le I^\circ \le S_2(\lambda) - S_1(\lambda),\ \lambda\in[\lambda_1,\lambda_2]\};
$$ 
the $(\widehat I, J)$-image of the flap base cut
along the curve $s=\{J=0,\ I<S_1(0)\}$ is shown in Fig.~\ref{fig/affine}, bottom right (here $\widehat I$ denotes a single-valued branch of the multi-valued action variable $I$ on the cutted flap base), and has the form 
$$
\mathbb R^2\setminus
\{(I,\lambda)\mid \max\{0,\lambda\}+S_1(\lambda)\le I\le S_2(\lambda),\ \lambda\in[\lambda_1,\lambda_2]\},
$$ 
for some continuous functions $S_1(\lambda)$ and $S_2(\lambda)$ such that $S_1(\lambda_1)=S_2(\lambda_1)$, 
$\lambda_2+S_1(\lambda_2)=S_2(\lambda_2)$, $\max\{0,\lambda\}+S_1(\lambda)<S_2(\lambda)$ for $\lambda_1<\lambda<\lambda_2$. Here $\alpha_1,\alpha_2$ are the parabolic orbits and $c$ is the elliptic-elliptic point, $J(\alpha_i)=\lambda_i$, $\lambda_1<J(c)=0<\lambda_2$.
One can show (see Remark \ref {rem:affine}) that, in the real-analytic case, $S_i(\lambda)$ are $C^1$-smooth on $[\lambda_1,\lambda_2]$ and $S_i''(\lambda)\to+\infty$ as $\lambda\to\lambda_i$.

\section{Acknowledgement}

The authors are grateful to Alexey Bolsinov for valuable remarks and discussions which improved the original
version of the article.  The work of E.K.\ was supported by Russian Science Foundation (grant No.~17-11-01303).

\bibliographystyle{amsplain}
\bibliography{./library.bib}

\end{document}